\documentclass{article}
\usepackage[T1]{fontenc}
\usepackage{lmodern}
\usepackage{geometry,amsmath,amssymb,mathrsfs,microtype}
\usepackage{enumitem,needspace}
\usepackage[hidelinks]{hyperref}
\hypersetup{pdftitle={Parahoric Spherical Whittaker Functions},
  pdfauthor={Markos Karameris},
  pdfkeywords={Whittaker functions, parahoric subgroups, Hecke algebras, Rankin--Selberg integrals}}
\newcommand{\GL}{\mathrm{GL}}
\newenvironment{proof}{\noindent\textbf{Proof\ }}{\hspace*{\fill}$\Box$\medskip}
\newtheorem{lemma}{Lemma}
\newtheorem{proposition}{Proposition}
\newtheorem{corollary}{Corollary}
\newtheorem{definition}{Definition}
\newtheorem{remark}{Remark}
\newtheorem{theorem}{Theorem}
\title{Parahoric Spherical Whittaker Functions}
\author{Markos Karameris}
\date{}
\begin{document}
\maketitle

\begin{abstract}
Let $K_\mu$ be a standard parahoric subgroup of $\GL_n(F)$ over a
non-archimedean local field. We express the diagonal values of normalized
parahoric Whittaker eigenfunctions with trivial central character by
blockwise Casselman--Shalika formulas. Every irreducible generic
Iwahori-spherical representation has a one-dimensional fixed space for
a suitable parahoric and no invariants under a strictly larger one.
The full Hecke character of this line determines the representation;
we compute its adjacent-block eigenvalues. Under additional
irreducibility hypotheses, an opposite induced model and
Demazure--Whittaker operators give an explicit formula on every Bruhat
cell. Applications to $\GL_n\times\GL_1$ and compatible parahorics in
$\GL_n\times\GL_m$ yield local Rankin--Selberg integrals and explicit
test vectors for their $L$-factors.
\end{abstract}

\section{Introduction and Notation}

Explicit Whittaker functions provide local test vectors for automorphic
$L$-functions. The spherical formulas of Shintani and Casselman--Shalika
\cite{shint,casselman-shalika} describe the maximal compact fixed vector,
while Baruch and the author \cite{bar-kar} compute the Iwahori-fixed
Whittaker function of a Steinberg representation. We study the
intermediate standard parahoric subgroups.

Theorem~\ref{thm:diag} expresses diagonal values as a product of Schur
Laurent polynomials, one for each Levi block. For a composition $\mu$,
Proposition~\ref{prop:ferrers-fixed-line} constructs a generic representation
induced from Steinberg factors of lengths $(\mu^+)^t$ with a
one-dimensional $K_\mu$-fixed space and no invariants under a strictly
larger parahoric. Every irreducible generic Iwahori-spherical representation
occurs. Its full parahoric Hecke character and canonical partition $\mu$
determine it uniquely (Corollary~\ref{cor:parahoric-character-classification}),
realizing the segment classification through one-dimensional parahoric
Hecke modules inside the Iwahori-fixed modules.
Under the hypotheses of Section~\ref{sec:opposite-induced-model}, an opposite
induced model and block symmetrization of the Iwahori formulas of
Brubaker--Buciumas--Bump--Gustafsson \cite{bbbg-whittaker} give the all-cell
formula of Theorem~\ref{thm:all-bruhat-cells}.
Reeder \cite[Proposition~4.5]{reeder-whittaker} computes the diagonal
Whittaker values of $\Phi_{J,+}$ and $\Phi_{J,-}$ and notes that
computing values at $aw$ with $w\ne1$ can be difficult
\cite[Introduction, p.~11]{reeder-whittaker}.
Section~\ref{sec:reeder-sections} extends both formulas to every Bruhat cell.

Section~\ref{sec:rankin-selberg-gl1} evaluates the $\GL_n\times\GL_1$
integral. Section~\ref{sec:rankin-selberg-gln-glm} treats $n>m$ when the
larger parahoric appends singleton blocks. After reduction to equal rank,
the lowest Laurent coefficient, conductor, and Rankin--Selberg functional
equation \cite{jpss} determine the integral and translated test vectors,
without multiplying the Whittaker formulas throughout the cell sum.

Let $F$ be a non-archimedean local field and let $|\cdot|$ be the normalized absolute value on $F$. We denote by $\mathcal{O}$ its ring of integers, by $\mathfrak{p}=\varpi\mathcal{O}$ its maximal ideal, and by $\mathfrak{f}=\mathcal{O}/\mathfrak{p}$ its residue field, of cardinality $q$. Let $G=\GL_n(F)$ and consider the following subgroups of $G$: the center $\mathcal{Z}$, consisting of the scalar matrices $z(t)$ with $t\in F^\times$; the maximal compact subgroup $K=\GL_n(\mathcal{O})$; the standard Borel subgroup $B$ of upper triangular matrices; the maximal unipotent subgroup $N$ of upper triangular unipotent matrices; and $N_{\mathcal{O}}=N\cap K$.

Let $T$ be the diagonal torus. For $\bar{k}=(k_1,\ldots,k_n)\in\mathbb Z^n$, write
$$
\begin{aligned}
d(t_1,\ldots,t_n)&=\operatorname{diag}(t_1,\ldots,t_n),\\
\varpi^{\bar{k}}=d_{\bar{k}}&=d(\varpi^{k_1},\ldots,\varpi^{k_n}).
\end{aligned}
$$
Put $T^+=\{\varpi^{\bar{k}}:k_1\geq\cdots\geq k_n\}$ and $T^0=T\cap K$.
Let $\mathbf{W}\simeq S_n$ be the Weyl group, generated by the simple reflections $s_1,\ldots,s_{n-1}$, and let $w_0$ be its longest element. All Weyl representatives are unsigned permutation matrices. Let $J$ be the standard Iwahori subgroup, namely the inverse image of the upper triangular Borel subgroup of $\GL_n(\mathfrak{f})$ under the reduction map $K\to \GL_n(\mathfrak{f})$. More generally, $J_m$ denotes the standard Iwahori subgroup of $\GL_m(F)$.

Throughout the paper, $\operatorname{Ind}$ denotes normalized smooth parabolic induction. We write $\nu=|\cdot|$ for the normalized absolute-value character of $F^\times$.

The root system is $\Phi=\{\alpha_{i,j}:i\neq j\}$, where
$\alpha_{i,j}(d(t_1,\ldots,t_n))=t_i/t_j$. Its positive, negative, and simple roots are
$$
\Phi^+=\{\alpha_{i,j}:i<j\},\qquad \Phi^-=-\Phi^+,\qquad
\Delta=\{\alpha_i=\alpha_{i,i+1}:1\leq i<n\}.
$$
For the matrix unit $e_{i,j}$, put $x_{i,j}(t)=I+te_{i,j}$, also denoted
$x_\alpha(t)$ for $\alpha=\alpha_{i,j}$. Every $d\in T$ has a unique
expression $d=\varpi^{\lambda(d)}d_0$, with $\lambda(d)\in\mathbb Z^n$
and $d_0\in T^0$. Then
$$
\begin{aligned}
dx_\alpha(t)d^{-1}&=x_\alpha\bigl(\alpha(d)t\bigr)
=x_\alpha\bigl(\alpha(d_0)t\varpi^{\langle\alpha,\lambda(d)\rangle}\bigr),\\
wx_\alpha(t)w^{-1}&=x_{w\alpha}(t).
\end{aligned}
$$

We use both root-theoretic and composition notation for standard parahorics.
For an ordered composition $\mu=(\mu_1,\ldots,\mu_r)$ of $n$, put
$$
m_0=0,\qquad m_a=\mu_1+\cdots+\mu_a,\qquad
C_a=\{m_{a-1}+1,\ldots,m_a\}\quad(1\leq a\leq r).
$$
Define
$$
S_\mu=
\Delta\setminus\{\alpha_{m_1},\ldots,\alpha_{m_{r-1}}\}.
$$
Each block with $\mu_a>1$ contributes a connected component of $S_\mu$ of size $\mu_a-1$. Set
$$
\begin{aligned}
W_\mu=W_{S_\mu}&=\langle s_i:\alpha_i\in S_\mu\rangle
\simeq S_{\mu_1}\times\cdots\times S_{\mu_r},\\
K_\mu=K_{S_\mu}&=\bigsqcup_{w\in W_\mu}JwJ.
\end{aligned}
$$
Equivalently, $K_\mu$ is the inverse image in $K$ of the standard parabolic subgroup of $\GL_n(\mathfrak f)$ with Levi block sizes $(\mu_1,\ldots,\mu_r)$.

Conversely, every subset $S\subseteq\Delta$ determines a unique ordered composition $\mu(S)$ by declaring the omitted simple roots to be the block boundaries. When $S$ is fixed, we write
$$
W_S=W_{\mu(S)},\qquad
\Phi_S=\langle S\rangle,\qquad
K_S=K_{\mu(S)}.
$$
When convenient, we identify $S$ with the index set $\{i:\alpha_i\in S\}$.

\section{Auxiliary Decompositions}\label{sec:decompositions}

Normalize Haar measure by $\operatorname{vol}(K_S)=1$ and let
$\mathcal H(G,K_S)$ be the algebra of compactly supported,
$K_S$-bi-invariant functions. Put $T_g=\mathbf1_{K_SgK_S}$ and
$e_S=\operatorname{vol}(K_S)^{-1}\mathbf1_{K_S}=\mathbf1_{K_S}$,
the spherical idempotent in $\mathcal H(G,J)$.
For a right $K_S$-invariant function $f$, right convolution gives
$$
(T_gf)(x)=\int_{K_SgK_S}f(xh)\,dh
=\sum_{i=1}^m f(x\gamma_i),
\qquad K_SgK_S=\bigsqcup_{i=1}^m\gamma_iK_S,
$$
since each right coset has volume $1$.

\Needspace{5\baselineskip}
The following standard Iwahori double-coset decomposition will be used below.

\begin{lemma}\label{cosrep}
Let $w\in\mathbf{W}$. Fix an ordering of the inversion set
$$
\Phi(w)=\{\alpha\in\Phi^+:w^{-1}\alpha\in\Phi^-\}.
$$
Then $|\Phi(w)|=\ell(w)$ and
$$
JwJ=\bigsqcup_{(t_\alpha)\in\mathfrak{f}^{\ell(w)}}
\left(\prod_{\alpha\in\Phi(w)}x_\alpha(\widetilde{t}_\alpha)\right)wJ,
$$
where $\widetilde{t}_\alpha\in\mathcal{O}$ is any fixed lift of $t_\alpha\in\mathfrak{f}$ and the product is taken in the chosen order.
\end{lemma}

Using the Bruhat decomposition of $K$ modulo $K_S$ together with the Iwasawa decomposition gives the following refinement.

\begin{lemma}\label{decomp}
Let $\mathbf{W}^S$ be the set of minimal-length representatives of $\mathbf{W}/W_S$. Then
$$
G=\bigsqcup_{\lambda\in\mathbb{Z}^n}\bigsqcup_{w\in\mathbf{W}^S}
N\varpi^\lambda wK_S.
$$
\end{lemma}

The next lemma identifies the negative-root elements in the parahoric.

\begin{lemma}\label{contain}
For $\alpha\in\Phi^-$,
$$
x_\alpha(t)\in K_S
\quad\Longleftrightarrow\quad
t\in
\begin{cases}
\mathcal O,&\alpha\in\Phi_S^-,\\
\mathfrak p,&\alpha\in\Phi^-\setminus\Phi_S^-.
\end{cases}
$$
The second case is equivalent to $W_S\alpha\subseteq\Phi^-$.
\end{lemma}

\begin{proof}
Put $P_S(\mathfrak f)=\operatorname{red}(K_S)$.
Since $x_\alpha(F)\cap K=x_\alpha(\mathcal O)$, take $t\in\mathcal O$
and write $\bar t=t+\mathfrak p$. Then
$x_\alpha(t)\in K_S\iff x_\alpha(\bar t)\in P_S(\mathfrak f)
\iff (\alpha\in\Phi_S^-\text{ or }\bar t=0)$.
As $\bar t=0\iff t\in\mathfrak p$, the criterion follows.
\end{proof}

\begin{definition}
We call $d\in T$ fundamental if there exists $i\in\{1,\ldots,n-1\}$ such that
$$
\langle\alpha_j,\lambda(d)\rangle=\delta_{ij}
$$
for every $j\in\{1,\ldots,n-1\}$. Thus, up to multiplication by an element of $T^0\mathcal{Z}$, the fundamental elements are
$$
d_i=\varpi^{(\underbrace{1,\ldots,1}_{i},0,\ldots,0)}.
$$
\end{definition}

For $w\in\mathbf{W}$, write $d^w=wdw^{-1}$.

\begin{proposition}\label{diag_decomp}
Let $d=z d_i$ with $z\in\mathcal Z$ and $1\leq i<n$, and let
$$
\mathcal{O}_S(d)=\{d^w:w\in W_S\}
$$
be its set of distinct $W_S$-conjugates. For each $d'\in\mathcal{O}_S(d)$, let $\mathcal{R}(d')$ be a set of representatives for
$$
N_{\mathcal O}/N_{\mathcal O}(d'),\qquad
N_{\mathcal O}(d'):=N_{\mathcal O}\cap d'Kd'^{-1}.
$$
Then
$$
K_SdK_S=\bigsqcup_{d'\in\mathcal{O}_S(d)}
\bigsqcup_{u\in\mathcal{R}(d')}ud'K_S.
$$
\end{proposition}

\begin{proof}
Let $w_S$ be the maximum length element of $W_S$. Note that for any $\alpha>0$, either $W_S\alpha>0$ or $w_S\alpha<0$. Now
$$
\begin{aligned}
K_SdK_S&=K_Sd^{w_S^{-1}}K_S=JW_SJd^{w_S^{-1}}K_S\\
&=N_{\mathcal O}W_SJd^{w_S^{-1}}K_S.
\end{aligned}
$$
If $\alpha>0$, then
$$
w_S^{-1}d^{-1}w_S
x_{-\alpha}(\varpi t)
w_S^{-1}dw_S =
x_{-\alpha}
\left(
t\varpi^{1+\langle w_S\alpha,\lambda(d)\rangle}
\right),
$$
with $t\in\mathcal O$. If $w_S\alpha>0$, then
$\langle w_S\alpha,\lambda(d)\rangle\geq0$, so the parameter lies in
$\mathfrak p$. Otherwise $\alpha\in\Phi_S^+$ and
$\langle w_S\alpha,\lambda(d)\rangle\geq-1$ by the minuscule property
of $d_i$, so the parameter lies in $\mathcal O$.
Lemma~\ref{contain} gives membership in $K_S$ in both cases.
Thus
$$
\left(d^{w_S^{-1}}\right)^{-1}N_{\mathfrak p}^{-}d^{w_S^{-1}}
\subseteq K_S,
\qquad
K_SdK_S=N_{\mathcal O}W_SN_{\mathcal O}d^{w_S^{-1}}K_S.
$$
For $\alpha>0$ and $t\in\mathcal O$, there are again two cases.
If $W_S\alpha>0$, then
$w_sx_\alpha(t)w_s^{-1}=x_{w_s\alpha}(t)\in N_{\mathcal O}$
for every $w_s\in W_S$. If $w_S\alpha<0$, then
$$
w_S^{-1}d^{-1}w_Sx_\alpha(t)w_S^{-1}dw_S
=x_\alpha\left(t\varpi^{-\langle w_S\alpha,\lambda(d)\rangle}\right)
\in N_{\mathcal O},
$$
since $\langle-w_S\alpha,\lambda(d)\rangle\geq0$.
The root-subgroup factorization gives
$
N_{\mathcal O} =
\mathcal N_{w_S}^{+}\mathcal N_{w_S}^{-},
$
where
$
\mathcal N_{w_S}^{+} =
\left\langle
x_\alpha(\mathcal O)
\right\rangle_{
\substack{\alpha\in\Phi^+\\ w_S^{-1}\alpha>0}
}
$
and
$
\mathcal N_{w_S}^{-} =
\left\langle
x_\alpha(\mathcal O)
\right\rangle_{
\substack{\alpha\in\Phi^+\\ w_S^{-1}\alpha<0}
}.
$
The two cases give
$$
W_S\mathcal N_{w_S}^{+}=\mathcal N_{w_S}^{+}W_S,
\qquad
\mathcal N_{w_S}^{-}d^{w_S^{-1}}\subseteq d^{w_S^{-1}}N_{\mathcal O}.
$$
Consequently,
$$
\begin{aligned}
K_SdK_S
&=N_{\mathcal O}W_S\mathcal N_{w_S}^{+}\mathcal N_{w_S}^{-}
d^{w_S^{-1}}K_S\\
&\subseteq N_{\mathcal O}\mathcal N_{w_S}^{+}W_Sd^{w_S^{-1}}N_{\mathcal O}K_S
=N_{\mathcal O}W_Sd^{w_S^{-1}}K_S\\
&=N_{\mathcal O}W_SdK_S
=\bigcup_{w_s\in W_S}N_{\mathcal O}d^{w_s}K_S.
\end{aligned}
$$
The reverse inclusion follows from $N_{\mathcal O},W_S\subseteq K_S$.
Finally, for $x_1,x_2\in N_{\mathcal O}$,
$$
x_1d^{w_s}K_S=x_2d^{w_s}K_S
\quad\Longleftrightarrow\quad
x_2^{-1}x_1\in N_{\mathcal O}\cap d^{w_s}K_S(d^{w_s})^{-1}.
$$
Since $K_S\cap N=N_{\mathcal O}$ and diagonal conjugation preserves
$N$, this stabilizer is $N_{\mathcal O}(d^{w_s})$.
Distinct elements of $\mathcal O_S(d)$ have distinct valuation vectors;
the uniqueness of the diagonal valuation in the Iwasawa decomposition
therefore makes their $N_{\mathcal O}$-orbits disjoint. Taking the
representatives $\mathcal R(d')$ gives the asserted disjoint union.
\end{proof}

\section{Support and Diagonal Values}

Fix a nontrivial additive character $\psi_F:F\to\mathbb{C}^\times$ that is trivial on $\mathcal{O}$ and nontrivial on $\mathfrak{p}^{-1}$. We use the corresponding nondegenerate character of $N$, also denoted by $\psi$, defined by
$$
\psi(n)=\psi_F\left(\sum_{i=1}^{n-1}n_{i,i+1}\right).
$$
Let $W$ be a right $K_S$-invariant Whittaker function, so that
$$
W(ng)=\psi(n)W(g)
$$
for $n\in N$ and $g\in G$.

For a root $\beta$, write
$$
W_S\beta<0
$$
to mean that $w_s\beta\in\Phi^-$ for every $w_s\in W_S$.

\begin{lemma}\label{supportlemma}
Let $d\in T$ and $w\in\mathbf{W}^S$. Then $W(dw)=0$ unless, for every $\alpha\in\Delta$,
$$
\langle\alpha,\lambda(d)\rangle\geq
\begin{cases}
-1,&W_Sw^{-1}\alpha<0,\\
0,&\text{otherwise}.
\end{cases}
$$
\end{lemma}

\begin{proof}
Fix $\alpha\in\Delta$ and choose $y\in\mathfrak{p}^{-1}$ such that $\psi_F(y)\neq1$. Put $\beta=w^{-1}\alpha$. Then
$$
\psi_F(y)W(dw)
=W(x_\alpha(y)dw)
=W\left(dw\,x_\beta\left(\alpha(d)^{-1}y\right)\right).
$$
The parameter on the right has valuation
$$
-1-\langle\alpha,\lambda(d)\rangle.
$$
If $\beta\in\Phi^+$, then $x_\beta(t)\in K_S$ for every $t\in\mathcal{O}$. If $\beta\in\Phi_S^-$, the same conclusion follows from Lemma~\ref{contain}. Thus, in either of these cases, if
$$
\langle\alpha,\lambda(d)\rangle<0,
$$
the rightmost root element belongs to $K_S$, forcing $W(dw)=0$.

It remains to consider $\beta\in\Phi^-\setminus\Phi_S^-$. This is equivalent to $W_S\beta<0$. By Lemma~\ref{contain}, $x_\beta(t)\in K_S$ whenever $t\in\mathfrak{p}$. Hence, if
$$
\langle\alpha,\lambda(d)\rangle<-1,
$$
the rightmost root element again belongs to $K_S$, and $W(dw)=0$.
\end{proof}

\begin{definition}
For $w\in\mathbf{W}^S$, define
$$
T_w^+=\left\{d\in T:
\langle\alpha,\lambda(d)\rangle\geq
\begin{cases}
-1,&W_Sw^{-1}\alpha<0,\\
0,&\text{otherwise}
\end{cases}
\text{ for every }\alpha\in\Delta\right\}.
$$
In particular,
$$
T_1^+=T^+T^0.
$$
\end{definition}

Retain the fundamental elements $d_i$ from Section~\ref{sec:decompositions}
and put $d_n=\varpi I_n$.

We consider an ordered partition
$$
n=n_1+\cdots+n_k.
$$
Put
$$
m_0=0,
\qquad
m_\ell=n_1+\cdots+n_\ell
$$
for $1\leq\ell\leq k$. The corresponding standard parahoric is determined by
$$
S=\{1,\ldots,n-1\}\setminus\{m_1,\ldots,m_{k-1}\}.
$$
For $f=(f_1,\ldots,f_n)$, write
$$
f^{(\ell)}=(f_{m_{\ell-1}+1},\ldots,f_{m_\ell}).
$$

Fix a right $K_S$-invariant Whittaker eigenfunction with trivial
central character, so $W(zg)=W(g)$ for $z\in\mathcal Z$, satisfying
$$
W\in\operatorname{Ind}_N^G(\psi),
\qquad W(1)=1,
\qquad T_{d_i}W=\lambda_iW\quad(1\leq i<n).
$$
Set $\Lambda_i=\lambda_iq^{-i(n-i)/2}$ for $1\leq i<n$ and
$\Lambda_0=\Lambda_n=1$, the latter because $d_n$ is central.
The boundary eigenvalues $\Lambda_{m_\ell}$ are nonzero: the invertible
Iwahori translation for $d_{m_\ell}$ commutes with the projection $e_S$
onto $K_S$-fixed vectors, because this diagonal is central in the Levi.
Its restriction to this subspace is therefore invertible and is a
nonzero scalar multiple of $T_{d_{m_\ell}}$.
For each block $1\leq\ell\leq k$, choose an $n_\ell$-tuple
$$
z^{(\ell)}=(z_{m_{\ell-1}+1},\ldots,z_{m_\ell})
$$
satisfying
$$
\sigma_j\left(z^{(\ell)}\right)
=\frac{\Lambda_{m_{\ell-1}+j}}{\Lambda_{m_{\ell-1}}},
\qquad 1\leq j\leq n_\ell.
$$
Here $\sigma_j$ denotes the $j$-th elementary symmetric polynomial, and we set $f_{m_k+1}=f_{n+1}=0$.

\begin{theorem}[Diagonal values]\label{thm:diag}
For $d=\varpi^f$, one has
$$
W(d) = 
\begin{cases}
\displaystyle
\delta_B^{1/2}(\varpi^f)
\prod_{\ell=1}^k
s_{f^{(\ell)}}\left(z^{(\ell)}\right),
&
d\in T_1^+,
\\[1em]
0,
&
d\notin T_1^+,
\end{cases}
$$
Here $s_\alpha$ is the Schur Laurent polynomial for a dominant integral signature $\alpha$; it is an ordinary Schur polynomial when $\alpha$ is a partition.
\end{theorem}

\begin{proof}
The vanishing outside $T_1^+$ follows from the support lemma, so it remains to compute the values for $d=\varpi^f\in T_1^+$.
Put $D_\ell=\sigma_{n_\ell}(z^{(\ell)})$.
Since $\prod_\ell D_\ell=\Lambda_n=1$, the Schur shift identity
$s_{a+c\mathbf1}(z)=(\prod_jz_j)^c s_a(z)$ shows that both sides are
centrally invariant. We may therefore assume $f_n=0$.
For $\epsilon\in\{0,1\}^n$ with $|\epsilon|=i$, first observe that
$$
\left|
N_{\mathcal O}/N_{\mathcal O}(\varpi^\epsilon)
\right| =
q^{\frac{i(n-i)}{2}}
\delta_B^{-1/2}(\varpi^\epsilon).
$$
Dominance gives $\varpi^fN_{\mathcal O}\varpi^{-f}\subseteq N_{\mathcal O}$,
so every unipotent representative contributes the character value $1$.
From Proposition~\ref{diag_decomp}, the relations $T_{d_i}(W)=\lambda_iW$,
and the definition $\widetilde W(d)=\delta_B^{-1/2}(d)W(d)$, we obtain
$$
\lambda_i\widetilde W(\varpi^f) =
q^{\frac{i(n-i)}{2}}
\sum_{\epsilon\in I_i^{n_\ell}}
\widetilde W(\varpi^{f+\epsilon}),
$$
where, for $i=m_{\ell-1}+j$ with $1\leq j\leq n_\ell$, we set
$$
I_i^{n_\ell} =
\left\{
\epsilon
\in
\{1\}^{m_{\ell-1}}
\times
\{0,1\}^{n_\ell}
\times
\{0\}^{n-m_\ell}
:
\sum_{r=m_{\ell-1}+1}^{m_\ell}\epsilon_r=j
\right\}.
$$
Equivalently (with $i=n$ supplied by central invariance),
$$
\Lambda_i\widetilde W(\varpi^f) =
\sum_{\epsilon\in I_i^{n_\ell}}
\widetilde W(\varpi^{f+\epsilon}).
$$
Starting with the first block $n_1$, put $c_1=f_{n_1+1}$ and define
$$
\widehat W_1(\varpi^x) =
\widetilde W
\left(
\varpi^{
(x_1+c_1,\dots,x_{n_1}+c_1,
f_{n_1+1},\dots,f_n)
}
\right).
$$
The first $n_1$ Hecke relations give
$$
\Lambda_j\widehat W_1(\varpi^x) =
\sum_{\substack{\eta\in\{0,1\}^{n_1}\\|\eta|=j}}
\widehat W_1(\varpi^{x+\eta}),
\qquad
1\leq j\leq n_1.
$$
Since $\sigma_j(z^{(1)})=\Lambda_j$, the usual Shintani argument gives
$$
\widehat W_1(\varpi^x) =
s_x\left(z^{(1)}\right)
\widehat W_1(1).
$$
Substituting
$
x=\left(
f_1-c_1,
\dots,
f_{n_1}-c_1
\right)
$ then gives
$$
\widetilde W(\varpi^f) =
s_{
f^{(1)} - f_{n_1+1}\mathbf 1_{n_1}
}
\left(z^{(1)}\right)
\widetilde W
\left(
\varpi^{
(
\underbrace{f_{n_1+1},\dots,f_{n_1+1}}_{n_1},
f_{n_1+1},f_{n_1+2},\dots,f_n
)
}
\right).
$$
If $k=1$, the first-block formula proves the theorem. For the second block ($k\geq2$), put $c_2=f_{m_2+1}$ and, for $x=(x_1,\dots,x_{n_2})$, define
$$
\widehat W_2(\varpi^x) =
\Lambda_{m_1}^{-x_1}
\widetilde W
\left(
\varpi^{
(
\underbrace{x_1+c_2,\dots,x_1+c_2}_{m_1},
x_1+c_2,x_2+c_2,\dots,x_{n_2}+c_2,
f_{m_2+1},\dots,f_n
)
}
\right).
$$
We claim that
$$
\frac{\Lambda_{m_1+j}}{\Lambda_{m_1}}
\widehat W_2(\varpi^x) =
\sum_{\substack{\eta\in\{0,1\}^{n_2}\\|\eta|=j}}
\widehat W_2(\varpi^{x+\eta}),
\qquad
1\leq j\leq n_2.
$$
Indeed, the relation corresponding to $d_{m_1+j}$ adds $1$ to all the first $m_1$ coordinates and to $j$ coordinates in the second block.
If $\eta_1=1$, the resulting diagonal is already of the form occurring in the definition of $\widehat W_2(\varpi^{x+\eta})$.
If $\eta_1=0$, the first $m_1$ coordinates have increased by $1$, while the first coordinate in the second block has not. In this case we use the relation corresponding to the block boundary $d_{m_1}$:
$$
\widetilde W(\varpi^g d_{m_1}) =
\Lambda_{m_1}\widetilde W(\varpi^g).
$$
If $x+\eta$ is non-dominant inside the second block, both corresponding
values vanish by Lemma~\ref{supportlemma}; the boundary shift changes
only the preceding block. Otherwise each summand equals
$\Lambda_{m_1}^{x_1+1}\widehat W_2(\varpi^{x+\eta})$.
Dividing by $\Lambda_{m_1}^{x_1+1}$ gives the claimed recurrence.
Since $\sigma_j(z^{(2)})=\Lambda_{m_1+j}/\Lambda_{m_1}$, the Shintani argument gives
$$
\widehat W_2(\varpi^x) =
s_x\left(z^{(2)}\right)
\widehat W_2(1).
$$
Taking
$
x =
\left(
f_{m_1+1}-f_{m_2+1},
\dots,
f_{m_2}-f_{m_2+1}
\right)
$
therefore gives
$$
\begin{aligned}
&
\widetilde W
\left(
\varpi^{
(
\underbrace{f_{m_1+1},\dots,f_{m_1+1}}_{m_1},
f_{m_1+1},f_{m_1+2},\dots,f_n
)
}
\right)
\\
&\qquad =
\Lambda_{m_1}^{
f_{m_1+1}-f_{m_2+1}
}
s_{
f^{(2)} - f_{m_2+1}\mathbf 1_{n_2}
}
\left(z^{(2)}\right)
\\
&\qquad\qquad\times
\widetilde W
\left(
\varpi^{
(
\underbrace{f_{m_2+1},\dots,f_{m_2+1}}_{m_2},
f_{m_2+1},f_{m_2+2},\dots,f_n
)
}
\right).
\end{aligned}
$$
Repeating the same argument inductively for the $\ell$-th block gives
$$
\begin{aligned}
&
\widetilde W
\left(
\varpi^{
(
\underbrace{f_{m_{\ell-1}+1},\dots,
f_{m_{\ell-1}+1}}_{m_{\ell-1}},
f_{m_{\ell-1}+1},
\dots,f_n
)
}
\right)
\\
&\qquad =
\Lambda_{m_{\ell-1}}^{
f_{m_{\ell-1}+1}-f_{m_\ell+1}
}
s_{
f^{(\ell)} - f_{m_\ell+1}\mathbf 1_{n_\ell}
}
\left(z^{(\ell)}\right)
\\
&\qquad\qquad\times
\widetilde W
\left(
\varpi^{
(
\underbrace{f_{m_\ell+1},\dots,
f_{m_\ell+1}}_{m_\ell},
f_{m_\ell+1},f_{m_\ell+2},\dots,f_n
)
}
\right).
\end{aligned}
$$
Repeating this for all blocks yields
$$
\widetilde W(\varpi^f) =
\prod_{\ell=1}^k
\Lambda_{m_{\ell-1}}^{
f_{m_{\ell-1}+1}-f_{m_\ell+1}
}
s_{
f^{(\ell)} - f_{m_\ell+1}\mathbf 1_{n_\ell}
}
\left(z^{(\ell)}\right).
$$
Since $D_\ell=\Lambda_{m_\ell}/\Lambda_{m_{\ell-1}}$ and $f_{n+1}=0$,
$$
\prod_{\ell=1}^k
\Lambda_{m_{\ell-1}}^{f_{m_{\ell-1}+1}-f_{m_\ell+1}}
=\prod_{\ell=1}^kD_\ell^{f_{m_\ell+1}},
\qquad
D_\ell^c s_{f^{(\ell)}-c\mathbf1}(z^{(\ell)})
=s_{f^{(\ell)}}(z^{(\ell)}).
$$
Substituting these identities and $W=\delta_B^{1/2}\widetilde W$
proves the formula.
\end{proof}

With $\Lambda_i=\lambda_iq^{-i(n-i)/2}$, Theorem~\ref{thm:diag}
has the usual Schur-polynomial form of Shintani's formula \cite{shint}.

\begin{remark}
For the standard parabolic $P=MN_P$ with
$M\simeq\prod_{\ell=1}^k\GL_{n_\ell}(F)$,
$\delta_B=\delta_P\delta_{B_M}$ rewrites Theorem~\ref{thm:diag} as
$$
W(d)=\delta_P^{1/2}(d)
\prod_{\ell=1}^k
\left(
\delta_{B_{n_\ell}}^{1/2}(d^{(\ell)})
s_{f^{(\ell)}}(z^{(\ell)})
\right).
$$
Each factor in parentheses is the normalized spherical Whittaker function
on the corresponding Levi block.
\end{remark}

\section{Representation-theoretic realization}

We fix the following segment convention. For an unramified character $\chi$ of $F^\times$ and an integer $m\geq1$, write
$$
\Delta(\chi,m)=
\left[
\chi\nu^{-(m-1)/2},
\chi\nu^{-(m-3)/2},
\ldots,
\chi\nu^{(m-1)/2}
\right].
$$
We order the characters from the lower endpoint to the upper endpoint. With this ordering,
$$
\operatorname{St}(\Delta(\chi,m))
=\operatorname{St}_m(\chi)
$$
is the unique generic irreducible quotient of
$$
\operatorname{Ind}_{B_m}^{\GL_m(F)}
\left(
\chi\nu^{-(m-1)/2},\ldots,
\chi\nu^{(m-1)/2}
\right).
$$
For an unramified character $\eta$ of $F^\times$, its standard local factor is
$$
L\left(s,\operatorname{St}_m(\chi)\otimes\eta\right)
=L\left(s+\frac{m-1}{2},\chi\eta\right).
$$
In particular, with the above ordering the standard $L$-factor is attached to the \emph{upper} endpoint $\chi\nu^{(m-1)/2}$ of the segment; see \cite[Theorem~8.2, with the second factor of rank $1$]{jpss}.

\begin{proposition}[Zelevinsky Classification]\label{prop:generic-iwahori-classification}
Let $\Delta_1,\ldots,\Delta_t$ be pairwise unlinked segments and let $\operatorname{St}(\Delta_i)$ be their essentially square-integrable representations. Then
$$
\operatorname{Ind}_{P}^{\GL_n(F)}
\left(
\operatorname{St}(\Delta_1)\otimes\cdots\otimes\operatorname{St}(\Delta_t)
\right)
$$
is irreducible and generic. Conversely, every irreducible generic representation of $\GL_n(F)$ with nonzero Iwahori-fixed vectors is obtained in this way, up to reordering of pairwise unlinked segments of unramified characters.
\end{proposition}
\begin{remark}
See \cite[Theorem~9.7(a),(b)]{zelevinsky} for the segment classification.
For the converse, the Borel--Casselman theorem realizes any irreducible
representation with nonzero Iwahori-fixed vectors as a subquotient of an
unramified principal series \cite[\S2.1, p.~531]{jantzen-iwahori}.
Thus its cuspidal support consists of unramified characters.
\end{remark}

Write $[a]_q=(q^a-1)/(q-1)$ for $a\geq0$, and put
$H_r=\mathcal H_q(S_r)$ over $\mathbb C$, with
$(T_s-q)(T_s+1)=0$. For a Young subgroup $W_\zeta$, its spherical
idempotent is $e_\zeta=(\sum_{w\in W_\zeta}q^{\ell(w)})^{-1}
\sum_{w\in W_\zeta}T_w$, in the standard finite-Hecke basis. These algebras are semisimple since $q>1$;
$S^\beta$ denotes the Specht module of shape $\beta$, regarded as a left
module via the anti-involution $T_w\mapsto T_{w^{-1}}$ when a source uses
right modules.

\begin{proposition}[Eigenvalues]\label{eigenval}
For $n,m\geq1$, put $N_0=n+m$ and $d=\min(n,m)$, and let $e_{n,m}$ be the spherical
idempotent of $\mathcal H_q(S_n\times S_m)$ in $H_{N_0}$.
Put
$$
A_n=[n]_q[m]_q\,e_{n,m}T_{s_n}e_{n,m}
\in e_{n,m}H_{N_0}e_{n,m}.
$$
For $0\leq j\leq d$, the space
$e_{n,m}S^{(N_0-j,j)}$ is a line, on which $A_n$ acts by
$$
\lambda_j=q^{j+1}[n-j]_q[m-j]_q-[j]_q.
$$
All other Specht modules have zero $e_{n,m}$-fixed space.
\end{proposition}

\begin{proof}
Let $G_0=\GL_{N_0}(\mathbb F_q)$, let $B_0$ be its upper triangular
Borel subgroup, and put $H=H_{N_0}$ and $e=e_{n,m}$.
Let $X_r$ be the Grassmannian of $r$-subspaces and $V_r=\mathbb C[X_r]$.
Write $X_*=\langle e_1,\ldots,e_n\rangle$ and
$P_0=\operatorname{Stab}_{G_0}(X_*)$.
Realize $H$ in $\mathbb C[G_0]$ by
$T_w\mapsto |B_0|^{-1}\sum_{g\in B_0wB_0}g$
\cite[\S2.5, p.~8]{mathas-cyclotomic}. Then
$$
e\longmapsto\frac1{|P_0|}\sum_{g\in P_0}g,\qquad
A_n\longmapsto a_n=\frac1{|P_0|}\sum_{g\in P_0s_nP_0}g.
$$
Indeed, the image of $eT_{s_n}e$ and $a_n$ are $P_0$-bi-invariant,
supported on the same double coset, and their coefficient sums are $q$ and
$|P_0s_nP_0/P_0|=q[n]_q[m]_q$, respectively.
If $\delta_X$ denotes the basis vector indexed by $X\in X_n$, the
$G_0$-equivariant Grassmann adjacency operator is
$$
\mathcal A_n\delta_X=
\sum_{\substack{Y\in X_n\\\dim(X\cap Y)=n-1}}\delta_Y,
\qquad a_n\delta_{X_*}=\mathcal A_n\delta_{X_*}.
$$

Averaging over the stabilizing parabolic gives
$V_r^{B_0}\simeq He_{r,N_0-r}$, the level-$1$ permutation module
$M^{(r,N_0-r)}$ of \cite[\S4.1, p.~28]{mathas-cyclotomic}.
Its Specht filtration (Corollary~4.6 there) splits by semisimplicity;
the multiplicities count semistandard tableaux of weight $(r,N_0-r)$.
These have at most two rows. For shape $(N_0-j,j)$, the bottom row
has $j$ twos and the top row is forced; such a tableau exists exactly
when $j\leq\min(r,N_0-r)$. Thus
$$
V_r^{B_0}\simeq
\operatorname{Ind}_{H_r\otimes H_{N_0-r}}^{H_{N_0}}\mathbf1
\simeq\bigoplus_{j=0}^{\min(r,N_0-r)}S^{(N_0-j,j)}.
$$
At $r=0,N_0$ this means the trivial module, with $H_0=\mathbb C$.
Frobenius reciprocity at $r=n$ now gives all the assertions about
the dimensions of $eS^\beta$.
Since $V_n\subset\operatorname{Ind}_{B_0}^{G_0}\mathbf1$, the finite-Hecke
correspondence also shows that $V_n$ is multiplicity-free.
For each irreducible $G_0$-constituent $E\subset V_n$,
the projection of $\delta_{X_*}$ onto $E$ is nonzero
(since $\delta_{X_*}$ generates $V_n$) and spans $E^{P_0}=eE^{B_0}$.
Projecting $a_n\delta_{X_*}=\mathcal A_n\delta_{X_*}$ onto $E$
therefore identifies the scalar of $\mathcal A_n$ on $E$ with that
of $A_n$ on $eE^{B_0}$.

To attach an eigenvalue to each shape, give $V_r$ the inner product
making its subspace basis orthonormal. Define $U_r:V_r\to V_{r+1}$
by summing over the $(r+1)$-subspaces containing a given $r$-subspace,
and put $D_{r+1}=U_r^*$ and $D_0=0$; these maps are $G_0$-equivariant.
Srinivasan constructs an orthogonal basis consisting of chains
$x_j,\ldots,x_{N_0-j}$ with $U_rx_r=x_{r+1}$
\cite[Theorem~1.2, equation~(8), p.~3, arXiv version]{srinivasan-grassmann}.
Orthogonality shows that $\ker D_j$ is spanned by the initial vectors
of chains starting at rank $j$. For $1\leq j\leq N_0/2$, the chain
lengths also show that $U_{j-1}$ is injective. Hence
$$
V_j=U_{j-1}V_{j-1}\mathbin{\oplus^\perp}\ker D_j,\qquad
(\ker D_j)^{B_0}\simeq S^{(N_0-j,j)}.
$$
The second identity follows by taking $B_0$-fixed vectors and comparing
the decompositions at ranks $j$ and $j-1$; it also holds for $j=0$.
For $0\leq j\leq d$, define
$$
E_{n,j}=U_{n-1}\cdots U_j(\ker D_j),
\qquad (E_{n,j})^{B_0}\simeq S^{(N_0-j,j)},
$$
where the product is the identity if $n=j$. The isomorphism follows
because every chain starting at $j$ reaches rank $n\leq N_0-j$,
so this product is injective on $\ker D_j$. This covers $n>m$ as well.

On such a chain, adjointness gives
$D_nx_n=(\|x_n\|^2/\|x_{n-1}\|^2)x_{n-1}$ for $n>j$.
Substituting total dimension $N_0$, starting rank $j$, and $u=n-1$
in Srinivasan's equation~(8), and using $D_jx_j=0$ when $n=j$, gives
$$
U_{n-1}D_n\big|_{E_{n,j}}
=q^j[n-j]_q[m-j+1]_q\operatorname{id}.
$$
The coefficient of an $n$-subspace $Y$ in $U_{n-1}D_nX$ counts
the $(n-1)$-subspaces of $X\cap Y$. It is $[n]_q$ when $X=Y$,
$1$ when $X,Y$ are adjacent, and $0$ otherwise. Hence
$\mathcal A_n=U_{n-1}D_n-[n]_q\operatorname{id}$. Its eigenvalue on
$E_{n,j}$, and hence that of $A_n$ on $eS^{(N_0-j,j)}$, is
$$
q^j[n-j]_q[m-j+1]_q-[n]_q
=q^{j+1}[n-j]_q[m-j]_q-[j]_q.
$$
\end{proof}

Let $\mu=(\mu_1,\ldots,\mu_r)$ be an ordered composition of $n$, and
let $K_\mu$ be the standard parahoric with these blocks.
Put $\lambda=(\mu^+)^t$, where $\mu^+$ is the decreasing
rearrangement of $\mu$. For unramified characters $\chi_1,\ldots,\chi_a$
whose segments $\Delta(\chi_j,\lambda_j)$ are pairwise unlinked, define
$$
\pi_{\lambda,\chi}=
\operatorname{Ind}_{P_\lambda}^{G}
\left(\operatorname{St}_{\lambda_1}(\chi_1)\otimes\cdots\otimes
\operatorname{St}_{\lambda_a}(\chi_a)\right),
$$
where $P_\lambda$ has Levi subgroup
$M_\lambda\simeq\prod_{j=1}^a\GL_{\lambda_j}(F)$.
This representation is irreducible and generic by the preceding classification.

\begin{proposition}\label{prop:ferrers-fixed-line}
The representation just defined satisfies
$$
\dim\pi_{\lambda,\chi}^{K_\mu}=1,\qquad
\pi_{\lambda,\chi}^{K'}=0\quad\text{for every parahoric }K'\supsetneq K_\mu.
$$
Thus its fixed line is new with respect to parahoric inclusion.
Let $K_S=K_\mu$ and let $v_S$ span this line. If $\alpha_i\notin S$
merges adjacent blocks of sizes $n,m$, put $d=\min(n,m)$.
For the adjacency generator
$A_i=T_{s_i}=\mathbf1_{K_Ss_iK_S}\in\mathcal H(K_{S\cup\{\alpha_i\}},K_S)$,
with $\operatorname{vol}(K_S)=1$, we have
$$
A_iv_S=-[d]_qv_S.
$$
\end{proposition}

\begin{proof}
Put $N=\mu_1+\cdots+\mu_r$. The Iwasawa compact picture and the
Steinberg sign character give
$$
\pi_{\lambda,\chi}^{J}\big|_{\mathcal H_q(S_N)}
\simeq
\operatorname{Ind}_{\mathcal H_q(W_\lambda)}^{\mathcal H_q(S_N)}
\operatorname{sgn}_\lambda;
$$
the unramified twists are trivial on the compact Levi.
Taking $K_S=K_\mu$-fixed vectors amounts to applying the spherical
idempotent of $\mathcal H_q(W_\mu)$. Under the Mackey decomposition,
double cosets are parametrized by matrices $A=(a_{ij})$ with row sums
$\mu$ and column sums $\lambda$. The contribution of $A$ to the fixed
space is nonzero precisely when the trivial character and the restricted
sign character agree on the intersection algebra
$$
\bigotimes_{i,j}\mathcal H_q(S_{a_{ij}}).
$$
This happens precisely when every $a_{ij}\leq1$, and each such
contribution is one-dimensional; compare the signed double-coset basis
in \cite[Lemma~3.1(3) and Proposition~5.5]{du-rui}.

Thus the dimension is the number of $0$--$1$ matrices with row sums
$\mu$ and column sums $\lambda=(\mu^+)^t$. For $0\leq h\leq a$,
$$
\sum_{j=1}^h a_{ij}\leq\min(h,\mu_i),\qquad
\sum_i\sum_{j=1}^h a_{ij}
=\sum_{j=1}^h\lambda_j=\sum_i\min(h,\mu_i).
$$
Thus $\sum_{j=1}^h a_{ij}=\min(h,\mu_i)$ for every $i,h$.
Subtracting the identities for $h$ and $h-1$ gives the unique Ferrers matrix
$$
F_{ih}=a_{ih}=\min(h,\mu_i)-\min(h-1,\mu_i)=
\begin{cases}
1,&h\leq\mu_i,\\
0,&h>\mu_i.
\end{cases}
$$
Therefore $\dim\pi_{\lambda,\chi}^{K_S}=1$; in the induced model, only
the Ferrers double coset supports a fixed section.

Let $K'\supsetneq K_\mu$ be parahoric. Its lattice chain is a proper
periodic subchain of that for $K_\mu$; hence, for some $g\in G$ and
composition $\zeta=(\zeta_1,\ldots,\zeta_s)$ of $N$,
$$
gK'g^{-1}=K_\zeta,\qquad s<r=\lambda_1.
$$
A nonzero Mackey contribution to $\pi_{\lambda,\chi}^{K_\zeta}$ would
require $A=(a_{ij})\in\{0,1\}^{s\times a}$ with column sums $\lambda$.
This is impossible, since
$$
\lambda_1=\sum_{i=1}^s a_{i1}\leq s<\lambda_1.
$$
Consequently $\pi_{\lambda,\chi}^{K'}\simeq
\pi_{\lambda,\chi}^{K_\zeta}=0$, proving newness.

The full parahoric Hecke algebra preserves this line, so acts by a character
$$
\vartheta_S:\mathcal H(G,K_S)\longrightarrow\mathbb C.
$$
To compute $\vartheta_S(A_i)$, suppose that $\alpha_i$ separates blocks
$p,p+1$, and put
$$
\mu_p=n,\qquad \mu_{p+1}=m,\qquad N_0=n+m,\qquad d=\min(n,m).
$$
Let $\mu'$ merge these blocks and write $K_i=K_{\mu'}$.
The local algebra $\mathcal H(K_i,K_S)$ embeds in $\mathcal H(K,K_S)$.
If $K^1$ is the principal congruence subgroup, reduction modulo
$\mathfrak p$ gives
$$
K_i/K^1\simeq P_{\mu'}(\mathfrak f),\qquad
K_S/K^1\simeq P_\mu(\mathfrak f).
$$
Let $N_{\mu'}$ be the unipotent radical of $P_{\mu'}$, and put
$L^{(i)}=\prod_{a\ne p,p+1}\GL_{\mu_a}(\mathfrak f)$.
Both parabolics contain $N_{\mu'}(\mathfrak f)$, and the quotient pair is
$$
\left(
P_{\mu'}(\mathfrak f)/N_{\mu'}(\mathfrak f),
P_\mu(\mathfrak f)/N_{\mu'}(\mathfrak f)
\right)
\simeq
\left(
L^{(i)}\times\GL_{N_0}(\mathfrak f),
L^{(i)}\times P_{n,m}(\mathfrak f)
\right).
$$
Quotienting by this common normal subgroup and cancelling the common
factor $L^{(i)}$ give, with each distinguished subgroup of volume $1$,
$$
\begin{aligned}
\mathcal H(K_i,K_S)
&\simeq\mathcal H(P_{\mu'}(\mathfrak f),P_\mu(\mathfrak f))\\
&\simeq\mathcal H(\GL_{N_0}(\mathfrak f),P_{n,m}(\mathfrak f)).
\end{aligned}
$$
On the finite Hecke side, let
$H_{n,m}=H_n\otimes H_m$ and let $e_{n,m}$ be its spherical idempotent.
The local algebra is the $q$-Johnson algebra
$e_{n,m}H_{N_0}e_{n,m}$. As in Proposition~\ref{eigenval}, the
characteristic function $A_i$ corresponds to the adjacency operator
$$
A_i\longleftrightarrow[n]_q[m]_q\,e_{n,m}T_{s_n}e_{n,m}.
$$
It remains to identify the module on which this algebra acts.

In the Mackey summand containing the Ferrers fixed vector, merging the
selected rows gives column-intersection sizes
$$
\rho=(2^d,1^{N_0-2d}).
$$
Writing $H_\rho=\bigotimes_jH_{\rho_j}$, the merged factor is
$$
Y^\rho=\operatorname{Ind}_{H_\rho}^{H_{N_0}}\operatorname{sgn}_\rho.
$$
Each unchanged block contributes a regular module whose spherical
projection is one-dimensional: if $e_t$ is the spherical idempotent of
$H_t$, then
$$
e_tH_t=\mathbb C e_t.
$$
Consequently, under the preceding Hecke-algebra isomorphism,
$$
\mathbb C v_S\simeq e_{n,m}Y^\rho.
$$

We now determine which spherical constituent occurs in $Y^\rho$.
Let $\mathbb D(V)$ denote the twist of an $H_{N_0}$-module by the
algebra involution
$$
T_s^\#=-qT_s^{-1}=q-1-T_s;
$$
thus $h$ acts on $\mathbb D(V)$ as $h^\#$ acts on $V$.
This interchanges the trivial and sign characters and commutes with
parabolic induction, so
$$
Y^\rho\simeq\mathbb D(M^\rho),\qquad
M^\rho=\operatorname{Ind}_{H_\rho}^{H_{N_0}}\mathbf1.
$$
The sign-twist duality for Specht modules
\cite[\S2 and Lemma~5.6]{kunzer-mathas-specht}, together with
self-duality of Specht modules in our semisimple setting, gives
$$
\mathbb D(S^\alpha)\simeq S^{\alpha^t}.
$$
The Specht filtration in \cite[Corollary~4.6]{mathas-cyclotomic}, at
Ariki--Koike level $1$, splits by semisimplicity and gives the ordinary
Young rule $[M^\rho:S^\beta]=K_{\beta,\rho}$. Here $K_{\beta,\rho}$
counts semistandard tableaux of shape $\beta$ and weight $\rho$. Hence
$$
[Y^\rho:S^\alpha]=K_{\alpha^t,\rho}.
$$

By Proposition~\ref{eigenval}, the irreducible modules
with nonzero $H_{n,m}$-fixed space are exactly
$S^{(N_0-j,j)}$, $0\leq j\leq d$, and their fixed spaces are lines.
Their multiplicities in $Y^\rho$ are
$$
[Y^\rho:S^{(N_0-j,j)}]
=K_{(2^j,1^{N_0-2j}),\,(2^d,1^{N_0-2d})}.
$$
For $j=d$ the two partitions agree, so the Kostka number is $1$.
For $j<d$, dominance fails at $j+1$, since
$$
\sum_{a=1}^{j+1}(2^j,1^{N_0-2j})_a
=2j+1<2j+2
=\sum_{a=1}^{j+1}(2^d,1^{N_0-2d})_a.
$$
Thus only $S^{(N_0-d,d)}$ contributes, and
$$
e_{n,m}Y^\rho\simeq e_{n,m}S^{(N_0-d,d)}.
$$
Finally Proposition~\ref{eigenval} gives
$$
A_iv_S=\lambda_dv_S
=\bigl(q^{d+1}[n-d]_q[m-d]_q-[d]_q\bigr)v_S
=-[d]_qv_S,
$$
because $d=\min(n,m)$. This covers both orders of the two block sizes.
\end{proof}

The full Hecke character of this line retains the entire representation.

\begin{corollary}\label{cor:parahoric-character-classification}
Let $\pi$ be irreducible, generic, and Iwahori-spherical. Let $\lambda$
be the decreasing list of its segment lengths and put $\mu=\lambda^t$.
The character $\vartheta_\pi$ of $\mathcal H(G,K_\mu)$ on
$\pi^{K_\mu}$, together with $\mu$, determines $\pi$ up to isomorphism.
As the unordered pairwise unlinked segments of unramified characters
vary, these pairs $(\mu,\vartheta_\pi)$ describe all such representations
without repetition.
\end{corollary}

\begin{proof}
Existence and exhaustion follow from Propositions~\ref{prop:generic-iwahori-classification}
and~\ref{prop:ferrers-fixed-line}. Put $H=\mathcal H(G,J)$ and
$e=e_{S_\mu}$. Then $eHe=\mathcal H(G,K_\mu)$ and
$e\pi^J=\pi^{K_\mu}\simeq\mathbb C_{\vartheta_\pi}$.
The simple $H$-module $\pi^J$ is a quotient of
$$
M=He\otimes_{eHe}\mathbb C_{\vartheta_\pi}.
$$
Since $eM\simeq\mathbb C_{\vartheta_\pi}$ generates $M$, every proper
submodule has zero $e$-image. Their sum is therefore the unique maximal
submodule, so $M$ has a unique simple quotient. The Iwahori equivalence
\cite[Theorem~2.1.1, p.~531]{jantzen-iwahori} then recovers $\pi$.
\end{proof}

This is a realization of the segment classification by the full
parahoric Hecke characters; the adjacent-block eigenvalues alone do not
determine the unramified twists. No additional irreducibility assumption
from Section~\ref{sec:opposite-induced-model} is used here.

\subsection{An opposite induced model for the fixed vector}
\label{sec:opposite-induced-model}

We construct a quotient map from an opposite induced representation
to a representation induced from Steinberg factors. Its identity-supported
section maps to the normalized parahoric fixed vector. Compatibility
with Whittaker functionals then lets us evaluate that vector in the
opposite model.

Let
$$
\lambda=(\lambda_1\geq\cdots\geq\lambda_r),
\qquad
\mu=\lambda^t=(\mu_1,\ldots,\mu_t),
\qquad
\mu_b=\#\{a:\lambda_a\geq b\}.
$$
The partition $\lambda$ specifies the Steinberg factors; $\mu$
specifies the parahoric $K_\mu$ and the Levi $M_\mu$. Write
$$
P_\lambda=M_\lambda N_\lambda,
\qquad
P_\mu=M_\mu N_\mu,
\qquad
\overline P_\mu=M_\mu\overline N_\mu,
$$
where
$$
M_\lambda\simeq\prod_{a=1}^{r}\GL_{\lambda_a}(F),
\qquad
M_\mu\simeq\prod_{b=1}^{t}\GL_{\mu_b}(F).
$$
Keep the characters $\chi_a$ and pairwise unlinked segments
$\Delta_a=\Delta(\chi_a,\lambda_a)$ from the definition of
$\pi_{\lambda,\chi}$, and write
$\tau_\lambda^-=\boxtimes_{a=1}^{r}\operatorname{St}_{\lambda_a}(\chi_a)$.
For comparison with Theorem~\ref{thm:diag}, we retain
$\prod_a\chi_a^{\lambda_a}=1$; the construction and all-cell formula
below also apply with arbitrary unramified central character.

For $1\leq a\leq r$ and $1\leq b\leq\lambda_a$, define
$$
\xi_{a,b}=\chi_a\nu^{(\lambda_a+1)/2-b}.
$$
Thus $\xi_{a,1}$ is the upper endpoint of $\Delta_a$.
We need two ordered tuples of these characters:
$$
\begin{aligned}
\boldsymbol\alpha
&=\left(
\xi_{1,\lambda_1},\ldots,\xi_{1,1},
\ldots,
\xi_{r,\lambda_r},\ldots,\xi_{r,1}
\right),\\
\boldsymbol\beta
&=\left(
\xi_{1,1},\ldots,\xi_{\mu_1,1},
\ldots,
\xi_{1,t},\ldots,\xi_{\mu_t,t}
\right).
\end{aligned}
$$
The tuple $\boldsymbol\alpha$ groups the characters by Steinberg
factor, in the order giving the Steinberg quotient. The tuple
$\boldsymbol\beta$ groups them by the blocks of $M_\mu$ and will
be the inducing tuple of the opposite model.

For any unramified tuple $\boldsymbol\eta$, write
$$
I(\boldsymbol\eta)=\operatorname{Ind}_B^G(\boldsymbol\eta).
$$
For each $a$, normalized induction gives the quotient
$$
q_a:
\operatorname{Ind}_{B_{\lambda_a}}^{\GL_{\lambda_a}(F)}
\left(\xi_{a,\lambda_a},\ldots,\xi_{a,1}\right)
\twoheadrightarrow\operatorname{St}_{\lambda_a}(\chi_a).
$$
Induction in stages therefore gives
$$
Q_\lambda
=\operatorname{Ind}_{P_\lambda}^{G}(q_1\boxtimes\cdots\boxtimes q_r):
I(\boldsymbol\alpha)\twoheadrightarrow\pi_{\lambda,\chi}.
$$

For $1\leq b\leq t$, let
$$
\sigma_b
=\operatorname{Ind}_{B_{\mu_b}}^{\GL_{\mu_b}(F)}
\left(\xi_{1,b},\ldots,\xi_{\mu_b,b}\right),
$$
and assume that every $\sigma_b$ is irreducible. Its Satake parameter is
$$
z^{(b)}
=\left(\xi_{1,b}(\varpi),\ldots,\xi_{\mu_b,b}(\varpi)\right).
$$
Choose its spherical vector $u_b^+$ and a Whittaker functional $\ell_b$
with $\ell_b(u_b^+)=1$, and put
$$
W_{\sigma_b}^{\circ}(g)
=\ell_b\left(\sigma_b(g)u_b^+\right).
$$
Irreducibility also gives the opposite-Borel realization
$$
\sigma_b\simeq
\operatorname{Ind}_{\overline B_{\mu_b}}^{\GL_{\mu_b}(F)}
\left(\xi_{1,b},\ldots,\xi_{\mu_b,b}\right).
$$
Fix this identification to carry the spherical section with value $1$
at the identity to $u_b^+$. Its normalized spherical Whittaker function
is then $W_{\sigma_b}^{\circ}$ in either realization.

We now specify the support of the fixed vector in $\pi_{\lambda,\chi}$.
Retain the column intervals $C_b$ and define
$$
l_0=0,\qquad l_a=\lambda_1+\cdots+\lambda_a,\qquad
R_a=\{l_{a-1}+1,\ldots,l_a\}.
$$
The permutation
$$
w_{\lambda,\mu}(m_{b-1}+a)=l_a-b+1
\qquad(1\leq a\leq\mu_b)
$$
reorders $\boldsymbol\beta$ into $\boldsymbol\alpha$. Our convention is
$$
(w\boldsymbol\eta)_i=\eta_{w^{-1}(i)},
\qquad
\boldsymbol\alpha=w_{\lambda,\mu}\boldsymbol\beta.
$$
Its explicit definition uses the labels $(a,b)$, so it is unambiguous
even when characters coincide. If $w_{0,\lambda}$ is the longest element
of $W_\lambda$, then
$$
x_{\lambda,\mu}=w_{0,\lambda}w_{\lambda,\mu},
\qquad
x_{\lambda,\mu}(m_{b-1}+a)=l_{a-1}+b.
$$
It is increasing on each $C_b$, and its inverse is increasing on each
$R_a$. Hence it is the minimal-length representative of
$W_\lambda w_{\lambda,\mu}W_\mu$.

Choose nonzero vectors
$$
u_a^-\in\operatorname{St}_{\lambda_a}(\chi_a)^{J_{\lambda_a}},
\qquad
u_\lambda^-=\bigotimes_a u_a^-,
$$
and let $\Lambda_\pi$ be any nonzero Whittaker functional on
$\pi_{\lambda,\chi}$. The quotient argument below will prove that
it is nonzero on the parahoric fixed line and determine its normalization.

\noindent\textbf{The Ferrers section.}
The unique Mackey summand in Proposition~\ref{prop:ferrers-fixed-line}
is represented by $x_{\lambda,\mu}$. Its incidence matrix gives
$$
\operatorname{pr}_{M_\lambda}
\left(P_\lambda\cap\dot x_{\lambda,\mu}K_\mu\dot x_{\lambda,\mu}^{-1}\right)
=\prod_a J_{\lambda_a}.
$$
Consequently the fixed line is spanned by the section $v_\mu$ defined by
$$
v_\mu(p\dot x_{\lambda,\mu}k)
=\delta_{P_\lambda}^{1/2}(p)\tau_\lambda^-(p)u_\lambda^-,
\qquad p\in P_\lambda,\quad k\in K_\mu,
$$
and zero elsewhere. The intersection fixes $u_\lambda^-$, so this is
well-defined. Since the inducing action is invertible and
$w_{0,\lambda}\in P_\lambda$, its support is exactly
$$
\operatorname{supp}(v_\mu)
=P_\lambda\dot x_{\lambda,\mu}K_\mu
=P_\lambda\dot w_{\lambda,\mu}K_\mu.
$$
This is support in the induced model; the Whittaker support is governed
by Lemma~\ref{supportlemma}.

\begin{lemma}[Whittaker Heredity]
\label{lem:whittaker-heredity-open}
Let $P=MN_P$ be a standard parabolic, let $\overline P=M\overline N_P$
be its opposite, and put $N_M=N\cap M$. Let $\sigma$ be a smooth
admissible representation of $M$, extended trivially across
$\overline N_P$, and let
$$
V=\operatorname{Ind}_{\overline P}^{G}(\sigma).
$$
Restriction to the open cell $\overline PN$ identifies
$$
\operatorname{Hom}_N(V,\psi)
\simeq
\operatorname{Hom}_{N_M}(\sigma,\psi|_{N_M}).
$$
More explicitly, the functional associated with
$\ell\in\operatorname{Hom}_{N_M}(\sigma,\psi|_{N_M})$ is characterized,
on sections compactly supported modulo $\overline P$ in
$\overline PN=\overline PN_P$, by
$$
\Lambda_\ell(F)
=\int_{N_P}\ell(F(u))\psi(u)^{-1}\,du.
$$
If $M=\prod_b GL_{m_b}(F)$ and
$\sigma=\boxtimes_b\sigma_b$ with each $\sigma_b$ irreducible and generic,
then this space of functionals has dimension one. If $\ell_b$ is a
Whittaker functional on $\sigma_b$, the corresponding Levi functional is
$$
\ell=\bigotimes_b\ell_b,\qquad
\ell\left(\bigotimes_b v_b\right)=\prod_b\ell_b(v_b).
$$
\end{lemma}

\begin{proof}
This is Rodier's heredity theorem
\cite[Theorems~1.4 and~1.6]{casselman-shalika}. For induction
from $\overline P$, the open-cell representative is $1$ and
its Levi character is $\psi|_{N_M}$, giving exactly the stated
integral and its unique extension. For a product Levi,
Whittaker uniqueness in each factor identifies the Levi
functional with $\bigotimes_b\ell_b$.
\end{proof}

Put
$$
\sigma_\mu=\boxtimes_b\sigma_b,\qquad
u_\mu^+=\bigotimes_bu_b^+,\qquad
\ell_{M_\mu}=\bigotimes_b\ell_b,\qquad
\mathcal V^{\mathrm{opp}}
=\operatorname{Ind}_{\overline P_\mu}^{G}(\sigma_\mu).
$$
Define $\Phi_\mu^\circ$ to be zero outside
$\overline P_\mu K_\mu$ and set
$$
\Phi_\mu^\circ(\bar p k)
=\delta_{\overline P_\mu}^{1/2}(\bar p)
\sigma_\mu(\bar p)u_\mu^+.
$$
This is well-defined because the Levi projection of
$\overline P_\mu\cap K_\mu$ is $M_\mu(\mathcal O)$: for every
$h$ in this intersection,
$$
\delta_{\overline P_\mu}^{1/2}(h)\sigma_\mu(h)u_\mu^+
=u_\mu^+.
$$
The parahoric factorization
$K_\mu=\overline N_\mu(\mathfrak p)M_\mu(\mathcal O)N_\mu(\mathcal O)$
gives
$$
\overline P_\mu K_\mu=\overline P_\mu N_\mu(\mathcal O),
\qquad
\Phi_\mu^\circ(u)=\mathbf1_{N_\mu(\mathcal O)}(u)u_\mu^+
\quad(u\in N_\mu).
$$
By Lemma~\ref{lem:whittaker-heredity-open}, $\ell_{M_\mu}$
induces $\Lambda_{\mathrm{opp}}$, given on compact open-cell
sections by
$$
\Lambda_{\mathrm{opp}}(F)
=\int_{N_\mu}\ell_{M_\mu}(F(u))\psi(u)^{-1}\,du.
$$
With $\operatorname{vol}(N_\mu(\mathcal O))=1$, the support
calculation and $\ell_b(u_b^+)=1$ yield
$$
\Lambda_{\mathrm{opp}}(\Phi_\mu^\circ)
=\prod_b\ell_b(u_b^+)=1.
$$

\Needspace{5\baselineskip}
The following rank-one intertwiners preserve the chosen Whittaker functional.

\begin{lemma}
\label{lem:rank-one-preservation}
Let $\boldsymbol\eta=(\eta_1,\ldots,\eta_n)$ be an ordered tuple
of unramified characters, with $I(\boldsymbol\eta)$ as above.
For $s=s_i$, the notation $s\boldsymbol\eta$ means that
$\eta_i,\eta_{i+1}$ are exchanged. With $\dot w_0$ the unsigned
permutation representative, use the standard Whittaker functional
$$
\Psi_{\boldsymbol\eta}(f)
=\int_N f(\dot w_0n)\psi(n)^{-1}\,dn
$$
and the standard rank-one operator
$$
\begin{aligned}
M_s(\boldsymbol\eta)&:I(\boldsymbol\eta)\longrightarrow
I(s\boldsymbol\eta),\\
(M_s(\boldsymbol\eta)f)(g)
&=\int_F f(\dot s^{-1}x_{\alpha_i}(x)g)\,dx.
\end{aligned}
$$
The integrals are interpreted by their standard continuation from
a convergence region. Put $z=(\eta_i/\eta_{i+1})(\varpi)$. Then
$$
\Psi_{s\boldsymbol\eta}\circ M_s(\boldsymbol\eta)
=\frac{1-q^{-1}z^{-1}}{1-z}\Psi_{\boldsymbol\eta}.
$$
The operator $(1-z)M_s(\boldsymbol\eta)$ is regular. Hence, whenever
$\eta_i/\eta_{i+1}\neq\nu$, the regularized operator
$$
R_s(\boldsymbol\eta)
=\frac{1-z}{1-q^{-1}z^{-1}}M_s(\boldsymbol\eta)
$$
is defined, including at $z=1$, and satisfies
$$
\Psi_{s\boldsymbol\eta}\circ R_s(\boldsymbol\eta)
=\Psi_{\boldsymbol\eta}.
$$
\end{lemma}

\begin{proof}
The first identity follows from
\cite[Proposition~3.7, equation~(10)]{bbbg-whittaker}, which gives
the factor $(1-q^{-1}z^{-1})/(1-z)$ for the lower-unipotent
Whittaker functional. Choose its character as
$\psi_-(n_-)=\psi(\dot w_0^{-1}n_-\dot w_0)$. Their functional
$\Omega^-_{\boldsymbol\eta}$ then satisfies
$\Psi_{\boldsymbol\eta}=\Omega^-_{\boldsymbol\eta}
\circ I(\boldsymbol\eta)(\dot w_0)$.
Since $M_s$ intertwines right translation by $\dot w_0$,
the same factor occurs for $\Psi_{\boldsymbol\eta}$.

To check regularity on the full induction, fix a compact open
$H\subset K$ and vary an $H$-fixed section $f_{\boldsymbol\eta}$
with its restriction to $K$ unchanged. Choose an open normal subgroup
$H_0\subset K$ contained in $H$. In the $(i,i+1)$ block,
$$
\begin{pmatrix}0&1\\1&x\end{pmatrix}
=\begin{pmatrix}-x^{-1}&1\\0&x\end{pmatrix}
\begin{pmatrix}1&0\\x^{-1}&1\end{pmatrix}.
$$
The last factor lies in $H_0$ for $|x|$ sufficiently large, uniformly
in $g\in K$. Thus for $|x|=q^k$ with $k$ sufficiently large,
right $H$-invariance and normalized induction give
$f_{\boldsymbol\eta}(\dot s^{-1}x_{\alpha_i}(x)g)
=q^{-k}z^k f_{\boldsymbol\eta}(g)$.
The set $|x|=q^k$ has measure $(1-q^{-1})q^k$, so the tail is
$$
(1-q^{-1})f_{\boldsymbol\eta}(g)\sum_{k\geq k_0}z^k
=\frac{(1-q^{-1})z^{k_0}}{1-z}f_{\boldsymbol\eta}(g).
$$
The remaining compact integral is Laurent polynomial in the parameters.
Thus $(1-z)M_s$ is regular on every $H$-fixed space.
Division by $1-q^{-1}z^{-1}$ gives $R_s$ when
$\eta_i/\eta_{i+1}\neq\nu$; at $z=1$ the tail formula gives
$R_s=\mathrm{id}$.
\end{proof}

\Needspace{5\baselineskip}
These intertwiners allow a regular reordering of the inducing characters.

\begin{lemma}
\label{lem:regular-reordering}
Assume that the segments $\Delta_a$ are pairwise unlinked and
$\lambda_1\geq\cdots\geq\lambda_r$. One can reorder
$w_0\boldsymbol\beta$ into
$\boldsymbol\alpha$ by adjacent exchanges for which
the operators of Lemma~\ref{lem:rank-one-preservation} are regular.
Their composite
$$
A:
I(w_0\boldsymbol\beta)
\longrightarrow I(\boldsymbol\alpha)
$$
satisfies
$$
\Psi_{\boldsymbol\alpha}\circ A
=\Psi_{w_0\boldsymbol\beta}.
$$
\end{lemma}

\begin{proof}
In $w_0\boldsymbol\beta$, the index $b$ decreases
and, for fixed $b$, the index $a$ decreases. For a fixed $a$,
the characters already occur as
$\xi_{a,\lambda_a},\ldots,\xi_{a,1}$, the order used for the
Steinberg quotient. Sort the labeled characters by increasing $a$,
preserving their order for each fixed $a$. Every adjacent exchanged pair
is therefore
$$
(\theta,\theta')=(\xi_{c,d},\xi_{a,b}),
\qquad a<c,\quad b\leq d.
$$
We must verify $\theta/\theta'\neq\nu$ for every such exchange.

Here $\nu(x)=|x|_F$, and all products below are products of
unramified characters of $F^\times$. If there is no integer $h$
such that $\xi_{a,1}=\xi_{c,1}\nu^h$, then
$\theta/\theta'\neq\nu$. Otherwise, unlinking means the segments
are disjoint and nonadjacent, or one contains the other.
In the first case no character from one segment has ratio $\nu$
with a character from the other. In the second case
$\lambda_a\geq\lambda_c$ implies $\Delta_c\subseteq\Delta_a$,
so for some integer $h\geq0$ we have
$$
\xi_{a,1}=\xi_{c,1}\nu^h,\qquad
\frac{\theta}{\theta'}
=\frac{\xi_{c,d}}{\xi_{a,b}}
=\nu^{-(h+d-b)}
\neq\nu.
$$
Thus Lemma~\ref{lem:rank-one-preservation} supplies a regular
operator for every exchange. Among the exchanged pairs, coincidence occurs
only in the containment case: $\xi_{c,d}=\xi_{a,b}$ precisely when $h=0$
and $b=d$ (same upper endpoint and same column).
Then $z=1$ and $R_s=\mathrm{id}$ by
Lemma~\ref{lem:rank-one-preservation}.

Write the successive tuples as
$\boldsymbol\eta^{(0)}=w_0\boldsymbol\beta$,
$\boldsymbol\eta^{(j)}=s_{i_j}\boldsymbol\eta^{(j-1)}$,
and $\boldsymbol\eta^{(L)}=\boldsymbol\alpha$, and put
$R_j=R_{s_{i_j}}(\boldsymbol\eta^{(j-1)})$.
Lemma~\ref{lem:rank-one-preservation} gives
$\Psi_{\boldsymbol\eta^{(j)}}\circ R_j
=\Psi_{\boldsymbol\eta^{(j-1)}}$ for each $j$.
Consequently $A=R_L\circ\cdots\circ R_1$ satisfies, by successive
substitution,
$$
\Psi_{\boldsymbol\alpha}\circ A
=\Psi_{\boldsymbol\eta^{(L)}}\circ R_L\circ\cdots\circ R_1
=\Psi_{\boldsymbol\eta^{(0)}}
=\Psi_{w_0\boldsymbol\beta}.
$$
\end{proof}

We now transfer the opposite-model Whittaker functional to the Steinberg-induced representation.

\begin{lemma}
\label{lem:opposite-quotient}
Under the hypotheses above, there is a surjective $G$-map
$$
\mathscr Q:
\mathcal V^{\mathrm{opp}}
\twoheadrightarrow\pi_{\lambda,\chi}
$$
which can be normalized so that
$$
\Lambda_\pi\circ\mathscr Q
=\Lambda_{\mathrm{opp}}.
$$
For this normalization,
$$
\mathscr Q\Phi_\mu^\circ
=\frac{v_\mu}{\Lambda_\pi(v_\mu)},
\qquad \Lambda_\pi(v_\mu)\neq0.
$$
For every $g\in G$,
$$
\boxed{
\frac{\Lambda_\pi(\pi_{\lambda,\chi}(g)v_\mu)}
{\Lambda_\pi(v_\mu)}
=
\Lambda_{\mathrm{opp}}
\left(\mathcal V^{\mathrm{opp}}(g)\Phi_\mu^\circ\right).
}
$$
\end{lemma}

\begin{proof}
Induction in stages identifies
$\mathcal V^{\mathrm{opp}}$
with $\operatorname{Ind}_{\overline B}^G
(\boldsymbol\beta)$. The map
$$
C:\operatorname{Ind}_{\overline B}^G
(\boldsymbol\beta)
\xrightarrow{\ \sim\ }I(w_0\boldsymbol\beta),
\qquad (CF)(g)=F(\dot w_0g),
$$
is $G$-equivariant and therefore preserves right $K_\mu$-invariance.
Explicitly,
$$
\Psi_{w_0\boldsymbol\beta}(CF)
=\int_N F(n)\psi(n)^{-1}\,dn.
$$
Here $\dot w_0^2=1$.
This left-translation map is distinct from the inverse-transpose
pullback of Lemma~\ref{lem:opposite-basis-transport}.
Compose $C$ with the regular reordering
of Lemma~\ref{lem:regular-reordering} and the Steinberg quotient:
$$
\mathcal V^{\mathrm{opp}}
\xrightarrow{C}I(w_0\boldsymbol\beta)
\xrightarrow{A}I(\boldsymbol\alpha)
\xrightarrow{Q_\lambda}\pi_{\lambda,\chi}.
$$
Since $Q_\lambda$ is surjective, heredity gives
$\Lambda_\pi Q_\lambda=a\Psi_{\boldsymbol\alpha}$
with $a\neq0$. Hence
$$
\Lambda_\pi Q_\lambda AC
=a\Psi_{w_0\boldsymbol\beta}C\neq0.
$$
The last functional is nonzero because $C$ is an isomorphism.
The Whittaker space of $\mathcal V^{\mathrm{opp}}$
is one-dimensional by Lemma~\ref{lem:whittaker-heredity-open};
rescaling the composite therefore gives
$\Lambda_\pi\mathscr Q=\Lambda_{\mathrm{opp}}$.
Its nonzero image is all of the irreducible $\pi_{\lambda,\chi}$.

Now
$$
\Lambda_\pi(\mathscr Q\Phi_\mu^\circ)
=\Lambda_{\mathrm{opp}}(\Phi_\mu^\circ)=1.
$$
Thus $\mathscr Q\Phi_\mu^\circ$ is a nonzero
$K_\mu$-fixed vector. Proposition~\ref{prop:ferrers-fixed-line}
therefore implies $\mathscr Q\Phi_\mu^\circ=c v_\mu$, with $c\Lambda_\pi(v_\mu)=1$.
This proves the normalization.
Applying $\Lambda_\pi\mathscr Q
=\Lambda_{\mathrm{opp}}$ to a right translate of
$\Phi_\mu^\circ$ proves the formula for every $g$.
\end{proof}

Write $W_\mu^\circ(g)$ for the common value in Lemma~\ref{lem:opposite-quotient};
then $W_\mu^\circ(1)=1$.

\begin{proposition}[Diagonal values]
\label{prop:diagonal-opposite-model}
For $d_f\in T^+$, write $f^{(b)}=f|_{C_b}$. Then
$$
\boxed{
W_\mu^\circ(d_f)
=\delta_B^{1/2}(d_f)\prod_{b=1}^{t}s_{f^{(b)}}(z^{(b)}).
}
$$
This is also the $w=1$ case of Theorem~\ref{thm:all-bruhat-cells}.
\end{proposition}

\begin{proof}
By Lemma~\ref{lem:opposite-quotient}, evaluate on $\Phi_\mu^\circ$.
For $d=d_f$ and $u\in N_\mu$,
$$
\Phi_\mu^\circ(ud)
=\delta_{P_\mu}^{-1/2}(d)
\mathbf1_{dN_\mu(\mathcal O)d^{-1}}(u)\sigma_\mu(d)u_\mu^+.
$$
Dominance makes $\psi$ trivial on this compact open-cell support,
whose volume is $\delta_{P_\mu}(d)$. Whittaker heredity therefore gives
$W_\mu^\circ(d)=\delta_{P_\mu}^{1/2}(d)
\prod_bW_{\sigma_b}^{\circ}(d^{(b)})$.
Apply the spherical Casselman--Shalika formula
\cite[Theorem~5.4]{casselman-shalika} in each block and use
$\delta_B(d)=\delta_{P_\mu}(d)\prod_b\delta_{B_{\mu_b}}(d^{(b)})$.
\end{proof}

\subsection{Whittaker values on every Bruhat cell}
\label{sec:all-bruhat-cells}

We retain the hypotheses and opposite induced model of
Section~\ref{sec:opposite-induced-model}: the segments $\Delta_a$
are pairwise unlinked, $\mu=\lambda^t$, and every $\sigma_b$ is
irreducible. Our Whittaker functions use the fixed character
$\psi$ of the upper unipotent group $N$; the auxiliary lower-unipotent
functional below makes the comparison with \cite{bbbg-whittaker} explicit.

Write the inducing tuple of that model as
$$
\boldsymbol\beta=(\beta_1,\ldots,\beta_n),
\qquad
\mathbf z_0=(\beta_1(\varpi),\ldots,\beta_n(\varpi)).
$$
The parameters are explicitly
$$
(\mathbf z_0)_{m_{b-1}+a}
=\chi_a(\varpi)q^{\,b-(\lambda_a+1)/2},
\qquad 1\leq a\leq\mu_b.
$$
Thus the entries indexed by $C_b$ are exactly the parameters
$z^{(b)}$ of the spherical representation $\sigma_b$.
We use $\mathbf z=(z_1,\ldots,z_n)$ for independent variables
until the final specialization $\mathbf z=\mathbf z_0$.

For an arbitrary unramified tuple $\boldsymbol\eta$ with
$\eta_i(\varpi)=z_i$, define the scalar opposite principal series
$$
\overline I(\boldsymbol\eta)
=\operatorname{Ind}_{\overline B}^{G}(\boldsymbol\eta).
$$
Thus its sections satisfy
$$
F(\bar b g)=\delta_{\overline B}^{1/2}(\bar b)
\boldsymbol\eta(\bar b)F(g),
\qquad
\bigl(\overline I(\boldsymbol\eta)(g)F\bigr)(h)=F(hg).
$$
To distinguish the compact subgroups in the two conventions, put
$$
J^+=J,\qquad K_\mu^+=K_\mu,\qquad
J^-=\operatorname{red}^{-1}(\overline B(\mathfrak f)),\qquad
K_\mu^-=\operatorname{red}^{-1}(\overline P_\mu(\mathfrak f)).
$$
The groups denoted by $J$ and $K_{S_\mu}$ in
\cite[Section~3 and Definition~4.2]{bbbg-whittaker} are $J^-$ and
$K_\mu^-$, respectively. The involution
$$
\iota(g)={}^{t}g^{-1}
$$
commutes with reduction and satisfies
$$
\begin{aligned}
\iota(B)&=\overline B,& \iota(N)&=N^-,&
\iota(J^+)&=J^-,& \iota(K_\mu^+)&=K_\mu^-,\\
\iota(\dot y)&=\dot y,&
\iota(d_f\dot w)&=d_{-f}\dot w.
\end{aligned}
$$
In particular, the involution exchanges the upper and lower compact
subgroups with the same ordered block sizes; it does not fix $J$ or
$K_\mu$.

\Needspace{5\baselineskip}
Pullback by the involution transports fixed sections between the two conventions.

\begin{lemma}
\label{lem:opposite-basis-transport}
The map
$$
\mathcal T_{\boldsymbol\eta}:
\overline I(\boldsymbol\eta)\longrightarrow
I(\boldsymbol\eta^{-1}),
\qquad
(\mathcal T_{\boldsymbol\eta}F)(h)=F(\iota(h)),
$$
is a linear isomorphism satisfying
$$
\mathcal T_{\boldsymbol\eta}\overline I(\boldsymbol\eta)(g)
=I(\boldsymbol\eta^{-1})(\iota(g))\mathcal T_{\boldsymbol\eta}.
$$
It identifies the $J^+$-fixed space with the $J^-$-fixed space,
and the $K_\mu^+$-fixed space with the $K_\mu^-$-fixed space.
Let $\Phi_y^{\mathbf z^{-1}}$ be the standard section in
\cite[Section~3, following Remark~3.2]{bbbg-whittaker}, and define
$$
\bar f_y^{\boldsymbol\eta}
=\mathcal T_{\boldsymbol\eta}^{-1}\Phi_y^{\mathbf z^{-1}}.
$$
These sections form a basis of $\overline I(\boldsymbol\eta)^J$,
with
$$
\operatorname{supp}(\bar f_y^{\boldsymbol\eta})
=\overline B\dot yJ,
\qquad
\bar f_y^{\boldsymbol\eta}(\bar b\dot yj)
=\delta_{\overline B}^{1/2}(\bar b)\boldsymbol\eta(\bar b).
$$
The section
$$
f_{1,\mu}^{\boldsymbol\eta}
=\sum_{y\in W_\mu}\bar f_y^{\boldsymbol\eta}
$$
is $K_\mu$-fixed, has value $1$ at the identity, and has support
$$
\overline BK_\mu
=\bigsqcup_{y\in W_\mu}\overline B\dot yJ.
$$
These identities hold for every unramified tuple.
\end{lemma}

\begin{proof}
For $b\in B$, one has
$\delta_{\overline B}(\iota(b))=\delta_B(b)$ and
$\boldsymbol\eta(\iota(b))=\boldsymbol\eta^{-1}(b)$.
Consequently
$$
(\mathcal T_{\boldsymbol\eta}F)(bh)
=\delta_B^{1/2}(b)\boldsymbol\eta^{-1}(b)
(\mathcal T_{\boldsymbol\eta}F)(h).
$$
The action identity follows from $\iota(hg)=\iota(h)\iota(g)$;
the fixed-space identities follow from the subgroup identities above.
The inverse is pullback by the same involution.

Transport the standard basis of
\cite[Section~3, following Remark~3.2]{bbbg-whittaker} and the
identity-coset parahoric section in the proof of
\cite[Proposition~4.3]{bbbg-whittaker}. Their support and summation
formulas give the remaining assertions, since $\iota(\dot y)=\dot y$.
\end{proof}

Use the upper-unipotent functional
$$
\Omega_{\boldsymbol\eta}(f)
=\int_N f(n)\psi(n)^{-1}\,dn,
$$
initially on compactly supported open-cell sections and then by
Whittaker heredity, or equivalently by the standard continuation
of the integral. Each positive-root coordinate has
$\operatorname{vol}(\mathcal O)=1$. Write
$$
\mathcal W_y(\mathbf z;g)
=\Omega_{\boldsymbol\eta}
\left(\overline I(\boldsymbol\eta)(g)\bar f_y^{\boldsymbol\eta}\right).
$$
The two indices play different roles: $y$ specifies the section,
whereas in $\mathcal W_y(\mathbf z;d_f\dot w)$ the element $w$
specifies the cell on which it is evaluated.

\Needspace{5\baselineskip}
The same involution transfers the Whittaker functionals with compatible measures.

\begin{lemma}
\label{lem:opposite-whittaker-transport}
Choose the lower-unipotent character and functional by
$$
\psi_-(n_-)=\psi(\iota(n_-)),
\qquad
\Omega^-_{\boldsymbol\eta^{-1}}(H)
=\int_{N^-}H(n_-)\psi_-(n_-)^{-1}\,dn_-.
$$
With $\operatorname{vol}(N(\mathcal O))
=\operatorname{vol}(N^-(\mathcal O))=1$, one has
$$
\Omega^-_{\boldsymbol\eta^{-1}}
\left(I(\boldsymbol\eta^{-1})(\iota(g))
\mathcal T_{\boldsymbol\eta}F\right)
=\Omega_{\boldsymbol\eta}
\left(\overline I(\boldsymbol\eta)(g)F\right).
$$
In the normalization of \cite[equation~(2)]{bbbg-whittaker}, this gives
$$
\boxed{
\mathcal W_y(\mathbf z;d_f\dot w)
=\delta_B^{1/2}(d_f)\,
\phi_y(\mathbf z;d_{-f}\dot w).
}
$$
\end{lemma}

\begin{proof}
The character $\psi_-$ has the conductor required in
\cite[Section~3]{bbbg-whittaker}. The involution carries
$N^-(\mathcal O)$ onto $N(\mathcal O)$, hence preserves the
normalized Haar measures. The change of variables $n=\iota(n_-)$ gives
$$
\begin{aligned}
&\Omega^-_{\boldsymbol\eta^{-1}}
\left(I(\boldsymbol\eta^{-1})(\iota(g))
\mathcal T_{\boldsymbol\eta}F\right)\\
&\quad=\int_{N^-}F(\iota(n_-)g)
\psi(\iota(n_-))^{-1}\,dn_-\\
&\quad=\int_N F(ng)\psi(n)^{-1}\,dn.
\end{aligned}
$$
This identity holds in the convergence region and extends with the
Whittaker integrals. By Lemma~\ref{lem:opposite-basis-transport},
$\mathcal T_{\boldsymbol\eta}\bar f_y^{\boldsymbol\eta}
=\Phi_y^{\mathbf z^{-1}}$. The definition in
\cite[equation~(2)]{bbbg-whittaker} uses this parameter
$\mathbf z^{-1}$ and multiplies the integral by
$\delta_B^{1/2}(d_{-f})=\delta_B^{-1/2}(d_f)$.
This proves the displayed comparison, with their weight equal to $f$
and their Weyl indices equal to $y,w$.
\end{proof}

Induction in stages expresses our fixed section as a sum of scalar Iwahori sections.

\begin{lemma}
\label{lem:scalar-fixed-section}
Realize each $\sigma_b$ as an opposite principal series and normalize
its spherical section to have value $1$ at the identity.
Under induction in stages,
$$
\mathcal V^{\mathrm{opp}}
\simeq\overline I(\boldsymbol\beta),
$$
the section $\Phi_\mu^\circ$ becomes
$$
f_\mu^\circ
=\sum_{y\in W_\mu}\bar f_y^{\boldsymbol\beta}.
$$
Moreover, $\Omega_{\boldsymbol\beta}(f_\mu^\circ)\neq0$
and for every $g\in G$,
$$
W_\mu^\circ(g)
=
\frac{\displaystyle\sum_{y\in W_\mu}
\mathcal W_y(\mathbf z_0;g)}
{\displaystyle\sum_{y\in W_\mu}
\mathcal W_y(\mathbf z_0;1)}.
$$
\end{lemma}

\begin{proof}
Induction in stages uses the chosen opposite principal-series
realizations of the $\sigma_b$. Its scalar section is obtained by
evaluating the Levi section at the identity. It has value $1$ on
$K_\mu$, and Levi Iwasawa decomposition gives
$$
\overline P_\mu=\overline B M_\mu(\mathcal O),
\qquad
\overline P_\mu K_\mu=\overline B K_\mu.
$$
Thus it is $f_{1,\mu}^{\boldsymbol\beta}$ of
Lemma~\ref{lem:opposite-basis-transport}, proving the section identity.

Proposition~3.6 of \cite{bbbg-whittaker} and
Lemma~\ref{lem:opposite-whittaker-transport}, at $y=w=1$ and $f=0$,
give $\Omega_{\boldsymbol\beta}(\bar f_1^{\boldsymbol\beta})=1$.
Whittaker heredity therefore makes $\Omega_{\boldsymbol\beta}$ a
nonzero multiple of the transported $\Lambda_{\mathrm{opp}}$.
Since $\Lambda_{\mathrm{opp}}(\Phi_\mu^\circ)=1$, the denominator
is nonzero. Normalize by it and apply
Lemma~\ref{lem:opposite-quotient}.
\end{proof}

We next give explicit operators for evaluating the individual
functions $\mathcal W_y$. Set $v=q^{-1}$ and let $s_i$ act on a
Laurent polynomial by exchanging $z_i,z_{i+1}$. Define
$$
\boxed{
\mathscr D_i F(\mathbf z)
=
\frac{(1-v)F(\mathbf z)
-(1-vz_{i+1}/z_i)F(s_i\mathbf z)}
{z_i/z_{i+1}-1}.
}
$$
All these operators act on the parameter variables, not on $g$.

This is the operator $\mathfrak T_i$ of
\cite[equation~(11)]{bbbg-whittaker}. The following two lemmas
record their results in the conventions established by
Lemma~\ref{lem:opposite-whittaker-transport}.

\begin{lemma}
\label{lem:demazure-recurrence}
The operators $\mathscr D_i$ preserve
$\mathbb C[z_1^{\pm1},\ldots,z_n^{\pm1}]$, satisfy the braid
relations, and satisfy
$$
\mathscr D_i^2=(v-1)\mathscr D_i+v,
\qquad
\mathscr D_i^{-1}=v^{-1}\mathscr D_i+(v^{-1}-1).
$$
For every $g\in G$,
$$
\mathcal W_{s_i y}(\mathbf z;g)
=
\begin{cases}
\mathscr D_i\mathcal W_y(\mathbf z;g),
&\ell(s_i y)=\ell(y)+1,\\
\mathscr D_i^{-1}\mathcal W_y(\mathbf z;g),
&\ell(s_i y)=\ell(y)-1.
\end{cases}
$$
In particular, if $y=s_{i_1}\cdots s_{i_a}$ is reduced, then
$$
\mathscr D_y=\mathscr D_{i_1}\cdots\mathscr D_{i_a}
$$
is independent of the reduced expression and
$\mathcal W_y=\mathscr D_y\mathcal W_1$.
\end{lemma}

\begin{proof}
These are \cite[equations~(11)--(13) and Theorem~3.8]{bbbg-whittaker}
transported by Lemma~\ref{lem:opposite-whittaker-transport}. The modular factor
is independent of $\mathbf z$, so it commutes with the operators.
Iteration gives $\mathcal W_y=\mathscr D_y\mathcal W_1$.
\end{proof}

The recurrence gives the following formula on each Iwahori cell.

\begin{lemma}
\label{lem:all-cell-initial}
For $w\in\mathbf W$ let
$$
\mathcal C_w=
\left\{f\in\mathbb Z^n:
f_i-f_{i+1}\geq
\begin{cases}
0,&w^{-1}(i)<w^{-1}(i+1),\\
-1,&w^{-1}(i)>w^{-1}(i+1)
\end{cases}
\text{ for every }i\right\}.
$$
For every $y\in\mathbf W$,
$$
\boxed{
\mathcal W_y(\mathbf z;d_f\dot w)
=
q^{-\ell(w)}\delta_B^{1/2}(d_f)
\mathscr D_y\mathscr D_w^{-1}(\mathbf z^f),
\qquad f\in\mathcal C_w.
}
$$
For $f\notin\mathcal C_w$, every
$\mathcal W_y(\mathbf z;d_f\dot w)$ is zero.
Here $\mathscr D_w^{-1}$ means the inverse of the operator
$\mathscr D_w$; for a reduced expression
$w=s_{i_1}\cdots s_{i_a}$ it is
$\mathscr D_{i_a}^{-1}\cdots\mathscr D_{i_1}^{-1}$.
\end{lemma}

\begin{proof}
Under Lemma~\ref{lem:opposite-whittaker-transport}, $\mathcal C_w$ is precisely
the $w$-almost-dominant cone of \cite[Definition~3.4]{bbbg-whittaker}.
The formula is their Corollary~3.9, using the path $w\to1\to y$;
its initial value at $y=w$ is their Proposition~3.6.
Vanishing is their Lemma~3.5, and their Corollary~3.10 gives
regular specialization at every parameter. Thus no dominance
condition beyond $f\in\mathcal C_w$ is imposed.
\end{proof}

Define the products over the positive roots of the Levi $M_\mu$ by
$$
\begin{aligned}
\Delta_\mu(\mathbf z)
&=\prod_{b=1}^{t}\prod_{m_{b-1}+1\leq i<j\leq m_b}
(1-z_j/z_i),\\
\mathfrak a_\mu(\mathbf z)
&=\prod_{b=1}^{t}\prod_{m_{b-1}+1\leq i<j\leq m_b}
(1-vz_j/z_i).
\end{aligned}
$$
Thus $i<j$ in every factor. Define the two operators
$$
\mathcal Z_\mu=\sum_{y\in W_\mu}\mathscr D_y,
\qquad
\mathcal P_\mu F
=\sum_{y\in W_\mu}y\left(\frac{F}{\Delta_\mu}\right).
$$
Here $\mathcal Z_\mu$ is the sum of the Demazure--Whittaker operators,
and $\mathcal P_\mu$ is the block Weyl symmetrizer.
The permutation action is
$(yF)(\mathbf z)=F(y^{-1}\mathbf z)$, with
$(y\mathbf z)_i=z_{y^{-1}(i)}$.

Block symmetrization combines the Iwahori sections and determines their normalization.

\begin{lemma}
\label{lem:block-symmetrizer}
For every Laurent polynomial $F$,
$$
\boxed{
\mathcal Z_\mu F
=\sum_{y\in W_\mu}\mathscr D_y F
=\mathfrak a_\mu\,\mathcal P_\mu F.
}
$$
The function $\mathcal P_\mu F$ is a Laurent polynomial and
$$
\mathcal P_\mu1=1,\qquad
\mathcal Z_\mu1=\mathfrak a_\mu.
$$
If each block signature $f^{(b)}$ is dominant, then
$$
\mathcal P_\mu(\mathbf z^f)
=\prod_{b=1}^{t}s_{f^{(b)}}(\mathbf z|_{C_b}).
$$
In particular,
$$
\Omega_{\boldsymbol\beta}(f_\mu^\circ)
=\mathfrak a_\mu(\mathbf z_0)\neq0.
$$
\end{lemma}

\begin{proof}
In \cite[equations~(10)--(11), p.~4]{cgp-demazure}, set the
covering degree to $1$ and $x^\gamma=\mathbf z^{-\gamma}$.
The Weyl action is then ordinary, and their operators become
$$
\begin{aligned}
D_iF&=\partial_iF=\frac{F-(z_{i+1}/z_i)s_iF}{1-z_{i+1}/z_i},\\
T_i&=(1-vz_{i+1}/z_i)\partial_i-1=\mathscr D_i.
\end{aligned}
$$
Let $\Phi_\mu^+$ be the positive Levi roots and put
$I(y^{-1})=\{\alpha\in\Phi_\mu^+:y^{-1}\alpha<0\}$.
The Weyl denominator satisfies
$$
y(\Delta_\mu)=(-1)^{\ell(y)}
\mathbf z^{\sum_{\alpha\in I(y^{-1})}\alpha}\Delta_\mu.
$$
Thus their Theorem~3, p.~5, applied to the commuting type-$A$
blocks, gives
$$
\partial_{w_{0,\mu}}F
=\frac{1}{\Delta_\mu}\sum_{y\in W_\mu}(-1)^{\ell(y)}
\mathbf z^{-\sum_{\alpha\in I(y^{-1})}\alpha}yF
=\sum_{y\in W_\mu}y\left(\frac{F}{\Delta_\mu}\right)
=\mathcal P_\mu F.
$$
Their Theorem~4 on the same page now gives
$\mathcal Z_\mu=\mathfrak a_\mu\mathcal P_\mu$.
These identities hold on rational functions.

Each $\partial_i$ preserves Laurent polynomials and fixes $1$.
The Weyl character formula gives the stated Schur product.
Finally Lemmas~\ref{lem:scalar-fixed-section}
and~\ref{lem:all-cell-initial}, at $f=0,w=1$, give
$$
\Omega_{\boldsymbol\beta}(f_\mu^\circ)
=(\mathcal Z_\mu1)(\mathbf z_0)=\mathfrak a_\mu(\mathbf z_0).
$$
Irreducibility of $\sigma_b$ excludes
$\beta_i/\beta_j=\nu^{\pm1}$ for distinct $i,j\in C_b$,
so every factor of this product is nonzero.
\end{proof}

\begin{theorem}[General Bruhat Cell Formula]
\label{thm:all-bruhat-cells}
Let $\lambda$ be a partition, let $\mu=\lambda^t$, and assume
that the segments $\Delta_a$ are pairwise unlinked and that
all principal series $\sigma_b$ are irreducible.
Let $\pi_{\lambda,\chi}$ and $\mathbf z_0$ be defined by the
inducing data above. The $K_\mu$-fixed line of $\pi_{\lambda,\chi}$ has a unique
Whittaker function $W_\mu^\circ$ normalized by $W_\mu^\circ(1)=1$.

For $f\in\mathbb Z^n$ and $w\in\mathbf W$, one has
$$
\boxed{
W_\mu^\circ(d_f\dot w)=
\begin{cases}
\displaystyle
\delta_B^{1/2}(d_f)\,q^{-\ell(w)}
\left(\mathcal P_\mu\mathscr D_w^{-1}\mathbf z^f\right)
(\mathbf z_0),
&d_f\in T_w^+,\\[4pt]
0,&d_f\notin T_w^+.
\end{cases}
}
$$
Here $T_w^+$ is the parahoric support cone of Definition~2:
$$
f_i-f_{i+1}\geq
\begin{cases}
-1,&W_\mu w^{-1}\alpha_i<0,\\
0,&\text{otherwise},
\end{cases}
\qquad 1\leq i<n.
$$
Equivalently, use $\mathcal P_\mu=\mathfrak a_\mu^{-1}\mathcal Z_\mu$.
Apply the operators to the monomial in independent variables before
specializing the parameters.
\end{theorem}

\begin{proof}
Proposition~\ref{prop:ferrers-fixed-line} gives the fixed line,
and Lemma~\ref{lem:opposite-quotient} gives its Whittaker normalization. Lemmas~\ref{lem:scalar-fixed-section}
and~\ref{lem:block-symmetrizer} then give, for every $g$,
$$
W_\mu^\circ(g)=\frac{1}{\mathfrak a_\mu(\mathbf z_0)}
\sum_{y\in W_\mu}\mathcal W_y(\mathbf z_0;g).
$$
For a minimal representative, vanishing outside $T_w^+$ is
Lemma~\ref{supportlemma}. For arbitrary $w=w_{\min}u$ with
$u\in W_\mu$, right $K_\mu$-invariance and
$$
W_\mu w^{-1}\alpha_i
=W_\mu w_{\min}^{-1}\alpha_i
$$
give $W_\mu^\circ(d_f\dot w)=W_\mu^\circ(d_f\dot w_{\min})$
and $T_w^+=T_{w_{\min}}^+$, proving the same vanishing assertion.

If $d_f\in T_w^+$ and $w^{-1}\alpha_i>0$, its orbit contains
a positive root and the parahoric lower bound is $0$.
If $w^{-1}\alpha_i<0$, that bound is $-1$ or $0$.
Thus $f\in\mathcal C_w$. Apply Lemma~\ref{lem:all-cell-initial}
to each summand and use
$\mathcal Z_\mu\mathscr D_w^{-1}\mathbf z^f
=\mathfrak a_\mu\mathcal P_\mu\mathscr D_w^{-1}\mathbf z^f$
from Lemma~\ref{lem:block-symmetrizer}. Cancelling
$\mathfrak a_\mu(\mathbf z_0)$ gives the stated formula.
The operators preserve Laurent polynomials and satisfy the
braid relations, so specialization is regular and the answer
is independent of the reduced word. The Bruhat--Iwahori
decomposition and the two equivariances recover all values on $G$.
\end{proof}

\begin{remark}[Special Cases]
For $w=1$, the theorem recovers Proposition~\ref{prop:diagonal-opposite-model}.
For $\mu=(n)$ and $\lambda=(1,\ldots,1)$, one has $K_\mu=K$;
the only minimal cell representative is $1$, giving the normalized
Casselman--Shalika formula
$$
W_{(n)}^\circ(d_f)=
\begin{cases}
\delta_B^{1/2}(d_f)s_f(\mathbf z_0),&f_1\geq\cdots\geq f_n,\\
0,&\text{otherwise}.
\end{cases}
$$
For $\mu=(1,\ldots,1)$ and $\lambda=(n)$, the representation is
$\operatorname{St}_n(\chi)$, $\mathcal P_\mu=1$, and
$$
(\mathbf z_0)_i=\chi(\varpi)q^{i-(n+1)/2},
\qquad (\mathbf z_0)_i/(\mathbf z_0)_{i+1}=v.
$$
The reflection coefficient in $\mathscr D_i$ vanishes at these parameters;
the inverse formula then gives, for every Laurent polynomial $F$,
$$
(\mathscr D_iF)(\mathbf z_0)
=(\mathscr D_i^{-1}F)(\mathbf z_0)=-F(\mathbf z_0).
$$
Thus $(\mathscr D_w^{-1}\mathbf z^f)(\mathbf z_0)
=(-1)^{\ell(w)}\mathbf z_0^f$, where
$\mathbf z_0^f=\chi(\det d_f)\delta_B^{1/2}(d_f)$, and
$$
W_{(1,\ldots,1)}^\circ(d_f\dot w)=
\begin{cases}
\chi(\det d_f)\delta_B(d_f)(-q)^{-\ell(w)},&d_f\in T_w^+,\\
0,&d_f\notin T_w^+.
\end{cases}
$$
This recovers \cite[Theorem~3]{bar-kar} and its unramified determinant
twists, with the full $w$-dominant support.
\end{remark}

\subsection{Reeder's two sections on all Bruhat cells}
\label{sec:reeder-sections}

Here $\mu$ is any composition and $\boldsymbol\eta$ is any unramified
tuple, with $z_i=\eta_i(\varpi)$.
Put $L=\ell(w_{0,\mu})$ and define
$$
f_{\mu,+}^\circ=\sum_{u\in W_\mu}\bar f_u^{\boldsymbol\eta},
\qquad
f_{\mu,-}^\circ=\sum_{u\in W_\mu}(-v)^{L-\ell(u)}
\bar f_u^{\boldsymbol\eta},
\qquad
W_{\mu,\pm}(\mathbf z;g)
=\Omega_{\boldsymbol\eta}\bigl(\overline I(\boldsymbol\eta)(g)
f_{\mu,\pm}^\circ\bigr).
$$
To distinguish Reeder's root subset from our Iwahori subgroup, write
$J_{\mathrm R}'=S_\mu$, $J_{\mathrm R}=-w_0S_\mu$,
$y_0=w_{0,J_{\mathrm R}}$, and $w_1=y_0w_0=w_0w_{0,\mu}$.
The $G$-isomorphism $CF(g)=F(\dot w_0g)$ from
$\overline I(\boldsymbol\eta)$ to $I(w_0\boldsymbol\eta)$ gives
$$
C\bar f_u^{\boldsymbol\eta}=\Phi_{w_0u}^{w_0\boldsymbol\eta},
\qquad
w_0u=yw_1\ (y\in W_{J_{\mathrm R}}),
\qquad \ell(y)=L-\ell(u).
$$
Thus $Cf_{\mu,\pm}^\circ=\Phi_{J_{\mathrm R},\pm}^{w_0\boldsymbol\eta}$
in \cite[\S4, p.~23]{reeder-whittaker}. Their Whittaker functions
agree with $W_{\mu,\pm}$, since $(CF)(\dot w_0ng)=F(ng)$.

\begin{remark}
\label{rem:reeder-plus}
Before specializing $\mathbf z=\mathbf z_0$ and dividing by
$\mathfrak a_\mu(\mathbf z_0)$, the proof of
Theorem~\ref{thm:all-bruhat-cells} gives
$$
W_{\mu,+}(\mathbf z;d_f\dot w)
=\delta_B^{1/2}(d_f)q^{-\ell(w)}\mathfrak a_\mu(\mathbf z)
\mathcal P_\mu(\mathscr D_w^{-1}\mathbf z^f)
\qquad(d_f\in T_w^+),
$$
and zero otherwise. At $w=1$, this and
Proposition~\ref{prop:reeder-minus} recover
\cite[Proposition~4.5]{reeder-whittaker}.
\end{remark}

Put $\rho_\mu=(\mu_1-1,\ldots,0;\ldots;\mu_t-1,\ldots,0)$ and
$\overline{\mathfrak a}_\mu=\prod_b\prod_{i<j\in C_b}(1-vz_i/z_j)$.
We define:
$$
\operatorname{Alt}_\mu F(\mathbf z)
=\sum_{u\in W_\mu}(-1)^{\ell(u)}F(u^{-1}\mathbf z).
$$
The sign section is evaluated by the following formula.

\begin{proposition}
\label{prop:reeder-minus}
For $f\in\mathcal C_w$,
$$
\boxed{
W_{\mu,-}(\mathbf z;d_f\dot w)
=\delta_B^{1/2}(d_f)q^{-\ell(w)}
\mathbf z^{w_{0,\mu}\rho_\mu}
\mathcal P_\mu\!\left(
\mathbf z^{-\rho_\mu}\overline{\mathfrak a}_\mu
\mathscr D_w^{-1}\mathbf z^f\right).
}
$$
For $f\notin\mathcal C_w$ the value is zero. The expression is a
Laurent polynomial in $\mathbf z$ and applies at every unramified parameter.
\end{proposition}

\begin{proof}
Put $R_\mu=\prod_b\prod_{i<j\in C_b}(z_i/z_j-1)$. We prove the
following identity of operators on rational functions:
$$
\begin{aligned}
\mathcal Z_{\mu,-}F
&:=\sum_{u\in W_\mu}(-v)^{L-\ell(u)}\mathscr D_uF\\
&=\frac{\operatorname{Alt}_\mu(\overline{\mathfrak a}_\mu F)}{R_\mu}
=\mathbf z^{w_{0,\mu}\rho_\mu}
\mathcal P_\mu(\mathbf z^{-\rho_\mu}\overline{\mathfrak a}_\mu F).
\end{aligned}
$$
Let $A_\mu F=\operatorname{Alt}_\mu(\overline{\mathfrak a}_\mu F)/R_\mu$.
First, both operators satisfy
$\mathcal Z_{\mu,-}(\mathscr D_i+1)=A_\mu(\mathscr D_i+1)=0$
for $\alpha_i\in S_\mu$. Indeed, partition $W_\mu$ into pairs
$u,us_i$ with $\ell(us_i)=\ell(u)+1$. Their contribution is
$$
(-v)^{L-\ell(u)}\mathscr D_u
+(-v)^{L-\ell(u)-1}\mathscr D_{us_i}
=(-v)^{L-\ell(u)-1}\mathscr D_u(\mathscr D_i-v),
$$
which annihilates $\mathscr D_i+1$ by the quadratic relation.
For $A_\mu$, put $r_i=z_{i+1}/z_i$ and use
$$
(\mathscr D_i+1)F=(1-vr_i)\frac{F-r_i s_iF}{1-r_i}.
$$
The quotient is $s_i$-invariant. Also $s_i$ permutes the factors
of $\overline{\mathfrak a}_\mu$ other than $1-v/r_i$, so
$\overline{\mathfrak a}_\mu(1-vr_i)$ is $s_i$-invariant.
Consequently $\overline{\mathfrak a}_\mu(\mathscr D_i+1)F$ is
$s_i$-invariant and its alternation is zero.

Next, compute in $\mathbb C(\mathbf z)\rtimes W_\mu$, where
$u h=u(h)u$ for a rational function $h$. For
$\mathcal L=\sum_uc_uu$, the coefficient of $u$ in
$\mathcal L(\mathscr D_i+1)$ is
$$
\frac{c_u\,u(1-vr_i)+c_{us_i}\,u(1-v/r_i)}{u(1-r_i)}.
$$
Setting this coefficient to zero determines $c_{us_i}$ from $c_u$,
since $u(1-v/r_i)$ is nonzero as a rational function.
The simple reflections connect every element of $W_\mu$ to
$w_{0,\mu}$; hence an operator with these annihilation identities
is determined by its $w_{0,\mu}$-coefficient.

Only $\mathscr D_{w_{0,\mu}}$ contributes to that coefficient in
$\mathcal Z_{\mu,-}$, since every other summand has smaller length.
Choose a reduced expression $w_{0,\mu}=s_{i_1}\cdots s_{i_L}$ and put
$\beta_k=s_{i_1}\cdots s_{i_{k-1}}\alpha_{i_k}$.
These are exactly the positive roots of the block Levi. Selecting the
reflection term in each factor of $\mathscr D_{w_{0,\mu}}$ gives
$$
\prod_{k=1}^{L}
\frac{-(1-v\mathbf z^{-\beta_k})}{\mathbf z^{\beta_k}-1}
=(-1)^L\frac{\mathfrak a_\mu}{R_\mu}.
$$
The $w_{0,\mu}$-coefficient of $A_\mu$ is
$(-1)^Lw_{0,\mu}(\overline{\mathfrak a}_\mu)/R_\mu$, which is the
same because $w_{0,\mu}(\overline{\mathfrak a}_\mu)=\mathfrak a_\mu$.
Thus $\mathcal Z_{\mu,-}=A_\mu$. Finally, put $\mathcal V_\mu=\prod_b\prod_{i<j\in C_b}(z_i-z_j)$.
Since $\Delta_\mu=\mathbf z^{-\rho_\mu}\mathcal V_\mu$ and
$u(\mathcal V_\mu)=(-1)^{\ell(u)}\mathcal V_\mu$, the definition
of $\mathcal P_\mu$ gives
$$
\mathcal P_\mu G
=\frac{\operatorname{Alt}_\mu(\mathbf z^{\rho_\mu}G)}{\mathcal V_\mu},
\qquad R_\mu=\mathbf z^{-w_{0,\mu}\rho_\mu}\mathcal V_\mu.
$$
Taking $G=\mathbf z^{-\rho_\mu}\overline{\mathfrak a}_\mu F$
proves the remaining equality. For $f\in\mathcal C_w$,
Lemma~\ref{lem:all-cell-initial} gives
$$
W_{\mu,-}(\mathbf z;d_f\dot w)
=\delta_B^{1/2}(d_f)q^{-\ell(w)}
\mathcal Z_{\mu,-}(\mathscr D_w^{-1}\mathbf z^f),
$$
so the operator identity proves the formula. The same lemma gives
vanishing outside $\mathcal C_w$; Laurent polynomial preservation by
$\mathscr D_i^{\pm1}$ gives regularity at every parameter.
\end{proof}

\section{The \texorpdfstring{$\GL_n\times \GL_1$}{GLn x GL1} Rankin--Selberg integral}
\label{sec:rankin-selberg-gl1}

Assume $n\geq2$, and let $W$ be a normalized right $K_S$-fixed
Whittaker eigenfunction as in Theorem~\ref{thm:diag}, where
$K_S=K_{(n_1,\ldots,n_k)}$. Write $z_1,\ldots,z_{n_1}$ for its
first-block parameters.
Let $\eta:F^\times\to\mathbb C^\times$ be unramified and normalize
$d^\times a$ by $\operatorname{vol}(\mathcal O^\times)=1$. For
$\operatorname{Re}(s)$ sufficiently large, set
$$
Z(s,W,\eta)
=
\int_{F^\times}
W\!\left(
\begin{pmatrix}
a&0\\
0&I_{n-1}
\end{pmatrix}
\right)
\eta(a)|a|^{s-\frac{n-1}{2}}\,d^\times a.
$$

\begin{proposition}[The $\GL_n\times \GL_1$ integral]
\label{prop:rankin-selberg-gl1}
With the notation above,
$$
\boxed{
Z(s,W,\eta)
=
\prod_{j=1}^{n_1}
\left(1-z_j\eta(\varpi)q^{-s}\right)^{-1}.
}
$$
\end{proposition}

\begin{proof}
Put $a_m=\operatorname{diag}(\varpi^m,1,\ldots,1)$.
Torus-unit invariance, the support lemma, and Theorem~\ref{thm:diag} give
$$
\begin{aligned}
Z(s,W,\eta)
&=\sum_{m\geq0}W(a_m)\eta(\varpi)^m q^{-m(s-(n-1)/2)}\\
&=\sum_{m\geq0}h_m(z_1,\ldots,z_{n_1})
\bigl(\eta(\varpi)q^{-s}\bigr)^m
=\prod_{j=1}^{n_1}(1-z_j\eta(\varpi)q^{-s})^{-1},
\end{aligned}
$$
where $h_m=s_{(m,0,\ldots,0)}$ is the complete homogeneous symmetric
polynomial and the last equality is its generating identity.
\end{proof}

For $\pi_{\lambda,\chi}$, the first-column parameters are the upper
endpoints $\xi_{a,1}=\chi_a\nu^{(\lambda_a-1)/2}$.
The same calculation, using Proposition~\ref{prop:diagonal-opposite-model}
and multiplicativity with the Steinberg factor
\cite[Theorem~8.2 and Proposition~9.4]{jpss}, gives
$$
Z(s,W_\mu^\circ,\eta)
=\prod_{a=1}^r(1-\xi_{a,1}(\varpi)\eta(\varpi)q^{-s})^{-1}
=L(s,\pi_{\lambda,\chi}\otimes\eta)
$$
for every unramified $\eta$, also when the central character is nontrivial.

\section{The \texorpdfstring{$\GL_n\times \GL_m$}{GLn x GLm} Rankin--Selberg integral}
\label{sec:rankin-selberg-gln-glm}

We compute the integral for $n>m$ when the larger parahoric is obtained
by appending singleton blocks. The all-cell formula determines its lowest
coefficient; the equal-rank functional equation then determines the integral.

\begingroup
\setlength{\emergencystretch}{2em}
\allowdisplaybreaks
\subsection{Representations and statement}

Retain the field, Whittaker character, parahoric, and normalized-induction
conventions above, writing $G_d,B_d,N_d,\psi_d$ in rank $d$ and $\times$ for
normalized induction. Fix $n>m$, put $e=n-m$, and let
$\mu=(\mu_1,\ldots,\mu_k)$ be an ordered composition of $m$. Set
$$
\mu^{(n)}=(\mu,1^e),\qquad
\lambda_a=\#\{b:\mu_b\ge a\},\qquad r=\max_b\mu_b,\qquad
\lambda^{(n)}=(\lambda_1+e,\lambda_2,\ldots,\lambda_r).
$$
For unramified characters $\chi_a,\eta_a$, consider
$$
\begin{aligned}
\pi_n&=\operatorname{St}_{\lambda_1+e}(\chi_1)
            \times\operatorname{St}_{\lambda_2}(\chi_2)\times\cdots\times\operatorname{St}_{\lambda_r}(\chi_r),\\
\pi_m&=\operatorname{St}_{\lambda_1}(\eta_1)
            \times\operatorname{St}_{\lambda_2}(\eta_2)\times\cdots\times\operatorname{St}_{\lambda_r}(\eta_r),\\
\tau&=\operatorname{St}_{\lambda_1}(\chi_1\nu^{e/2})
            \times\operatorname{St}_{\lambda_2}(\chi_2)\times\cdots\times\operatorname{St}_{\lambda_r}(\chi_r).
\end{aligned}
$$
For arbitrary ordered $\mu$ assume
$$
\chi_a/\chi_c,\ \eta_a/\eta_c\notin\nu^{\frac12\mathbb Z}\quad(a\ne c),
\qquad \nu^{\frac12\mathbb Z}=\{\nu^{j/2}:j\in\mathbb Z\}.
$$
This separates different rows of the half-integrally centered segments,
including those of $\tau$, so all three segment products and their column
principal series are irreducible. No condition between $\chi_a$ and $\eta_c$
is imposed. For decreasing $\mu$, the weaker hypotheses of
Theorem~\ref{thm:all-bruhat-cells} suffice separately for all three
representations: pairwise unlinked segments and irreducible column principal series.

Let the unique fixed Whittaker functions, each of value $1$ at the identity, be
$$
W_n\in\mathcal W(\pi_n,\psi_n)^{K_{\mu^{(n)}}},\qquad
W_m\in\mathcal W(\pi_m,\psi_m^{-1})^{K_\mu},\qquad
W_\tau\in\mathcal W(\tau,\psi_m)^{K_\mu}.
$$
For decreasing $\mu$, Lemma~\ref{lem:opposite-quotient} and
Theorem~\ref{thm:all-bruhat-cells} apply directly. For ordered $\mu$, number
each row's boxes by their occurrences from left to right. The reversed
column tuple has each row in Steinberg-quotient order; regrouping by rows
exchanges only characters from different rows. The stronger hypothesis
excludes ratios $\nu^{\pm1}$, so Lemma~\ref{lem:rank-one-preservation}
provides regular Whittaker-preserving intertwiners.
Proposition~\ref{prop:ferrers-fixed-line} already covers ordered compositions,
and the remaining symmetrizer and support arguments apply block by block.
Here we impose no restriction on the central characters, as allowed in
Section~\ref{sec:opposite-induced-model}.

Use the block endpoints $m_b$ and intervals $C_b$ of this composition,
and put $\kappa_i=m-m_b$ for $i\in C_b$. Set
$$
\mathcal L_\mu=\sum_{b<c}\mu_b\mu_c=\sum_i\kappa_i,
\qquad \ell_* =\sum_{b<c}\sum_{a=1}^{\mu_c}\min(a,\mu_b).
$$
For decreasing $\mu$, $\ell_*=\sum_b(b-1)\binom{\mu_b+1}{2}$.
Put $\gamma_a=\chi_a(\varpi)$, $\beta_a=\eta_a(\varpi)$ and
$r_a(b)=\#\{c\le b:\mu_c\ge a\}$. The block parameters are
$$
z_a^{(b)}=\gamma_a q^{r_a(b)-(\lambda_a^{(n)}+1)/2},\qquad
y_a^{(b)}=\beta_a q^{r_a(b)-(\lambda_a+1)/2},\qquad a\le\mu_b.
$$
Their concatenations are $\mathbf z=(z_1,\ldots,z_m)$ and
$\mathbf y=(y_1,\ldots,y_m)$. Thus $\mathbf z$ is both the first $m$
parameters of $\pi_n$ and the complete column tuple of $\tau$.
Write $\det z^{(b)}$ for the product of the entries of block $b$.
With $\operatorname{vol}(K_\mu)=\operatorname{vol}(N_m(\mathcal O))=1$
and the resulting quotient measure, define
$$
Z(s)=\int_{N_m\backslash G_m}
W_n(\operatorname{diag}(g,I_e))W_m(g)|\det g|^{s-e/2}\,dg.
$$

\begin{theorem}\label{thm:rs-gln-glm}
Put
$$
C_\mu(\mathbf z,\mathbf y)
=q^{-\ell_*}\prod_{b=1}^k
\left(\det z^{(b)}\det y^{(b)}\right)^{-(m-m_b)}.
$$
Then
$$
\boxed{Z(s)=C_\mu(\mathbf z,\mathbf y)q^{s\mathcal L_\mu}
L(s,\pi_n\times\pi_m).}
$$
The Euler factor is
$$
L(s,\pi_n\times\pi_m)=
\prod_{a,c=1}^r\prod_{j=1}^{\min(\lambda_a^{(n)},\lambda_c)}
\left(1-\gamma_a\beta_c
q^{-s-(\lambda_a^{(n)}+\lambda_c)/2+j}\right)^{-1}.
$$
The identity holds initially for $\operatorname{Re}(s)$ sufficiently large,
and then as an identity of rational functions in $q^{-s}$.
\end{theorem}

\subsection{Summing over Bruhat cells}

Use $\mathbf W^\mu$ for the minimal representatives in rank $m$, and let
$b(j)$ denote the block containing $j$. Write $\varepsilon_i(w)=1$ if
$b(w^{-1}(i))>b(w^{-1}(i+1))$, and $0$ otherwise. These are precisely the
indicators in the support condition of Theorem~\ref{thm:all-bruhat-cells}.
Put $X=q^{-s}$, $\varphi=\mathbf1_{\mathcal O^m}$, and $e_m=(0,\ldots,0,1)$.
The same theorem gives
\begin{equation}\label{eq:rs-restriction}
W_n(\operatorname{diag}(g,I_e))
=|\det g|^{e/2}W_\tau(g)\varphi(e_mg).
\end{equation}
Indeed, by Lemma~\ref{decomp} and equivariance it suffices to take
$g=d_fw$, $w\in \mathbf W^\mu$. The symmetrizer and Demazure operators act
only on the first $m$ variables, where the parameters agree. The support
conditions also agree for $i<m$; the root
$e_{w^{-1}(m)}-e_{m+1}>0$ adds exactly $f_m\ge0$.
The remaining appended exponents are zero, and
$$
\delta_{B_n}(\operatorname{diag}(d_f,I_e))
=\delta_{B_m}(d_f)|\det d_f|^e.
$$
For general $g=u d_fw k$, $e_mu=e_m$ and $e_mwk$ is a primitive integral
row, so $\varphi(e_mg)=\mathbf1_{\{f_m\ge0\}}$. This proves
\eqref{eq:rs-restriction} and identifies $Z(s)$ with the equal-rank integral
$$
Z(s)=\int_{N_m\backslash G_m}W_\tau(g)W_m(g)\varphi(e_mg)|\det g|^s\,dg.
$$

The finite Bruhat decomposition, with the inversion-root representatives
of Lemma~\ref{cosrep}, gives $q^{\ell(w)}$ right $K_\mu$-cosets $uwK_\mu$
for each minimal $w$, with $u\in N_m(\mathcal O)$. Each has volume $1$,
and the two Whittaker characters cancel in the product. Iwasawa integration, with torus units
of volume $1$, therefore gives
\begin{equation}\label{eq:rs-cell-sum}
Z(s)=\sum_{w\in \mathbf W^\mu}q^{\ell(w)}
\sum_{\substack{f_m\ge0\\f_i-f_{i+1}\ge-\varepsilon_i(w)}}
\delta_{B_m}^{-1}(d_f)W_\tau(d_fw)W_m(d_fw)X^{|f|},
\quad |f|=\sum_i f_i.
\end{equation}
The two modulus factors in the all-cell formula cancel
$\delta_{B_m}^{-1}(d_f)$.

\subsection{The lowest Laurent coefficient}

Set $\rho=(m-1,\ldots,0)$ and retain $\rho_\mu$ from
Section~\ref{sec:reeder-sections}. Then $\kappa=\rho-\rho_\mu$
is constant on each block.

The support conditions determine a unique lowest-degree contribution.

\begin{lemma}\label{lem:rs-minimum}
Every pair $(f,w)$ permitted by the support conditions in
\eqref{eq:rs-cell-sum} satisfies $|f|\ge-\mathcal L_\mu$.
There is exactly one such pair $(f_*,w_*)$ with equality, and
$\ell(w_*)=\ell_*$.
\end{lemma}
\begin{proof}
Put $h=f+\rho$. The support conditions give
$h_1\ge\cdots\ge h_m\ge0$; along equal heights the labels
$b(w^{-1}(i))$ must be strictly decreasing. Hence the $\mu_b$ heights
assigned to block $b$ are distinct nonnegative integers, and
$$
|f|=\sum_i h_i-\binom m2
\ge\sum_b\binom{\mu_b}{2}-\binom m2=-\mathcal L_\mu.
$$
Equality forces block $b$ to have heights $0,\ldots,\mu_b-1$.
Thus $h_*$ is the decreasing rearrangement of $\rho_\mu$, with equal
heights ordered by decreasing block label. Assigning successive
occurrences of block $b$ to successive indices of $C_b$ gives the unique
minimal $w_*$, and $f_*=h_*-\rho$ satisfies the support conditions.

For $b<c$, the entry of block $c$ at height $t$ precedes exactly
$\min(t+1,\mu_b)$ entries of block $b$. Summing over $0\le t<\mu_c$
and $b<c$ counts all inversions; entries in the same block create none.
This gives $\ell(w_*)=\ell_*$.
\end{proof}

Retain $v=q^{-1}$ and $\mathscr D_i,\mathcal P_\mu$ from Section~\ref{sec:all-bruhat-cells}, and set
$$
R_i=v\mathscr D_i^{-1}=\mathscr D_i+1-v,\qquad
\mathscr R_w=v^{\ell(w)}\mathscr D_w^{-1}
=R_{i_l}\cdots R_{i_1}
\quad(w=s_{i_1}\cdots s_{i_l}\text{ reduced}).
$$
The order is reversed because of the inverse. By
Section~\ref{sec:reeder-sections},
$\mathcal P_\mu F=\operatorname{Alt}_\mu(z^{\rho_\mu}F)/\mathcal V_\mu$,
with $\mathcal V_\mu$ as defined there.

The next identity evaluates the Whittaker functions at this extremal point.

\begin{lemma}\label{lem:rs-extremal}
With independent variables $z_1,\ldots,z_m$,
$$
\boxed{\mathcal P_\mu\mathscr R_{w_*}z^{f_*}
=(-v)^{\ell_*}z^{-\kappa}.}
$$
\end{lemma}
\begin{proof}
Conjugating by $z^\rho$ gives the degree-preserving polynomial operators
$$
\mathscr B_i=z^\rho R_i z^{-\rho},\qquad
\mathscr B_iF=
\frac{(1-v)z_iF-(z_i-vz_{i+1})s_iF}{z_i-z_{i+1}}.
$$
Write $\mathscr B_u$ in the usual reduced-word order. Expanding the
difference quotient shows that all monomial exponents have decreasing
rearrangement dominated by that of the input. Apart from the endpoints
$a,s_i a$, they are strictly dominated: two unequal coordinates are
replaced by intermediate ones with the same sum. Moreover, for
nonnegative $a$,
$$
\mathscr B_i z^a\equiv-vz^{s_i a}+(1-v)z^a
\quad(a_i>a_{i+1}),\qquad
\mathscr B_i z^a=-vz^a\quad(a_i=a_{i+1}),
$$
where the congruence is modulo monomials with strictly smaller decreasing
rearrangement. In particular the span of monomials strictly dominated
by $h_*$ is stable under every $\mathscr B_i$.

Let $H$ stabilize $h_*$, and let $d_0$ be the shortest permutation with
$d_0h_*=\rho_\mu$. It assigns equal-height positions to their targets
in increasing order. Thus $d_0$ is increasing on each $H$ block and
$d_0^{-1}$ is increasing within each $C_b$, so $d_0$ is minimal in
$W_\mu d_0H$. The reversed equal-height order defining $w_*$ gives
$$
w_*^{-1}=d_0w_{0,H},\qquad
\ell_* =\ell(d_0)+\ell(w_{0,H}),\qquad
\mathscr B_{w_{0,H}}z^{h_*}=(-v)^{\ell(w_{0,H})}z^{h_*}.
$$
Along a reduced word for $d_0$, each leading swap exchanges strictly
decreasing heights: successive suffixes $u,s_i u$ are minimal modulo
$H$, with $u^{-1}(i)<u^{-1}(i+1)$; equal source heights would give the
same $H$-coset and contradict minimality. The exchanged endpoint
therefore always has coefficient $-v$, giving
$$
\mathscr B_{d_0}z^{h_*}
=(-v)^{\ell(d_0)}z^{\rho_\mu}+\sum_x c_xz^{xh_*}+Q.
$$
Every other endpoint path omits a swap, so each $x$ has a representative
shorter than $d_0$; every monomial of $Q$ is strictly dominated by $h_*$.

All these terms have degree $\sum_b\binom{\mu_b}{2}=\deg\mathcal V_\mu$.
A nonnegative monomial of this degree survives $\operatorname{Alt}_\mu$
only if each block has exactly the exponents $0,\ldots,\mu_b-1$.
This excludes $Q$. For $z^{xh_*}$ it requires $x\in W_\mu d_0H$,
which excludes every shorter representative by the minimality of $d_0$. Consequently
$$
\operatorname{Alt}_\mu\bigl(\mathscr B_{w_*^{-1}}z^{h_*}\bigr)
=(-v)^{\ell_*}\operatorname{Alt}_\mu(z^{\rho_\mu})
=(-v)^{\ell_*}\mathcal V_\mu.
$$
Since $z^\rho\mathscr R_{w_*}z^{f_*}
=\mathscr B_{w_*^{-1}}z^{h_*}$ and $\kappa$ is block-constant,
$$
\mathcal P_\mu\mathscr R_{w_*}z^{f_*}
=z^{-\kappa}
\frac{\operatorname{Alt}_\mu(z^\rho\mathscr R_{w_*}z^{f_*})}{\mathcal V_\mu}
=(-v)^{\ell_*}z^{-\kappa}.
$$
\end{proof}

After removing the common factor $\delta_{B_m}^{1/2}(d_{f_*})$,
Theorem~\ref{thm:all-bruhat-cells} gives the values of $W_\tau$ and $W_m$
at $d_{f_*}w_*$ as $(-v)^{\ell_*}z^{-\kappa}$ and
$(-v)^{\ell_*}y^{-\kappa}$, respectively. Replacing $\psi_m$ by $\psi_m^{-1}$
does not change values at the points $d_fw$: the conjugating diagonal unit can
be moved to the right into $K_\mu$. Their product, including the cell
measure, is explicitly
$$
\begin{aligned}
q^{\ell_*}\delta_{B_m}^{-1}(d_{f_*})
W_\tau(d_{f_*}w_*)W_m(d_{f_*}w_*)
&=q^{\ell_*}(-v)^{2\ell_*}z^{-\kappa}y^{-\kappa}\\
&=C_\mu(\mathbf z,\mathbf y)\ne0.
\end{aligned}
$$
Thus the Laurent expansion begins with
\begin{equation}\label{eq:rs-leading-term}
Z(X)=C_\mu(\mathbf z,\mathbf y)X^{-\mathcal L_\mu}
+\text{terms of higher degree}.
\end{equation}

\subsection{The functional equation and completion of the proof}

\noindent\textbf{Proof of Theorem~\ref{thm:rs-gln-glm}.}
Multiplicativity and the Steinberg-pair formula
\cite[Theorem~8.2 and Proposition~9.4]{jpss} give the displayed Euler factor and
$$
L_0(s):=L(s,\tau\times\pi_m)=L(s,\pi_n\times\pi_m).
$$
Indeed, only the first row differs. Since $\lambda_1\ge\lambda_c$,
both first-row products run over $1\le j\le\lambda_c$, and
$$
(\chi_1\nu^{e/2})(\varpi)\beta_c
q^{-(\lambda_1+\lambda_c)/2+j}
=\gamma_1\beta_c q^{-(\lambda_1+e+\lambda_c)/2+j}.
$$
The ideal theorem \cite[Theorem~2.7(ii)]{jpss} gives
$P(X):=Z(s)/L_0(s)\in\mathbb C[X,X^{-1}]$.
Since $L_0^{-1}(0)=1$, equation~\eqref{eq:rs-leading-term} implies
$$
\min\deg P=-\mathcal L_\mu,\qquad
[X^{-\mathcal L_\mu}]P=C_\mu(\mathbf z,\mathbf y).
$$

For unramified $\alpha,\beta$, the conductor formula is
$a(\operatorname{St}_A(\alpha)\times\operatorname{St}_B(\beta))
=AB-\min(A,B)$; see
\cite[Introduction, paragraph~2; Section~1.3; Proposition in Section~1.4]
{bushnell-henniart-exponents} for compatibility, additivity, and this formula.
Consequently
$$
a(\tau\times\pi_m)
=m^2-\sum_{a,c}\min(\lambda_a,\lambda_c)
=m^2-\sum_b\mu_b^2=2\mathcal L_\mu,
$$
where $\sum_{a,c}\min(\lambda_a,\lambda_c)
=\sum_{j\ge1}\#\{a:\lambda_a\ge j\}^2$ and the column lengths of $\lambda$
are the decreasing rearrangement of $\mu$. As $\psi_F$ has conductor $\mathcal O$,
$$
\varepsilon(s,\tau\times\pi_m,\psi_F)=cX^{2\mathcal L_\mu},\qquad c\ne0.
$$

To obtain the upper degree bound, put
$$
\widetilde W(g)=W(w_0g^{-t}),\qquad
W^\natural(g)=\widetilde W(gw_0)=W(w_0g^{-t}w_0).
$$
For either of our $K_\mu$-fixed functions, $W^\natural$ has the opposite
Whittaker character and is fixed by $K_{\mu^{\rm rev}}$, where
$\mu^{\rm rev}=(\mu_k,\ldots,\mu_1)$, since
$w_0(K_{\mu^{\rm rev}})^{-t}w_0=K_\mu$.
Lemma~\ref{supportlemma} and Lemma~\ref{lem:rs-minimum} therefore give
the same lower bound $-\mathcal L_\mu$ for the dual integral.
Only the general support lemma is needed here.

With self-dual additive measure $\widehat\varphi=\varphi$. Simultaneous right
translation of both dual functions by $w_0$ preserves the integral,
because $\varphi(e_mgw_0)=\varphi(e_mg)$ and $|\det w_0|=1$.
Thus the equal-rank functional equation
\cite[Theorem~2.7(iii), equation~(2.7.1)]{jpss}, after division by the
two $L$-factors, reads
$$
\widetilde P(q^{-1}X^{-1})=cX^{2\mathcal L_\mu}P(X),
\qquad \min\deg\widetilde P\ge-\mathcal L_\mu.
$$
Here $\widetilde P(Y)$ is the dual integral divided by
$L(s',\tau^\vee\times\pi_m^\vee)$, $Y=q^{-s'}$; its reciprocal
$L$-factor has constant term $1$, so division preserves the lower bound.
The central-character factor at $-1$ is $1$ because the twists are unramified.
The left side has largest exponent at most $\mathcal L_\mu$, hence
$\max\deg P\le-\mathcal L_\mu$. Together with the lowest coefficient above this gives
$$
P(X)=C_\mu(\mathbf z,\mathbf y)X^{-\mathcal L_\mu},
$$
which proves the theorem.\hfill$\Box$

\subsection{Exact test vectors}

The theorem gives a generator of the Rankin--Selberg ideal.
A simultaneous translation and constant rescaling give the $L$-factor itself.

\begin{corollary}[Exact Test Vectors]\label{cor:rs-exact-test-vectors}
Put $h=d_{-\kappa}$, so $|\det h|=q^{\mathcal L_\mu}$, and define
$$
\begin{aligned}
W_n^\sharp(g)&=C_\mu(\mathbf z,\mathbf y)^{-1}q^{-e\mathcal L_\mu/2}
W_n\bigl(g\operatorname{diag}(h,I_e)\bigr),\\
W_m^\sharp(g)&=W_m(gh).
\end{aligned}
$$
Then $Z(s,W_n^\sharp,W_m^\sharp)=L(s,\pi_n\times\pi_m)$.
\end{corollary}
\begin{proof}
The substitution $g'=gh$ gives
$$
Z(s,W_n^\sharp,W_m^\sharp)
=C_\mu^{-1}q^{-e\mathcal L_\mu/2}|\det h|^{-s+e/2}Z(s)
=L(s,\pi_n\times\pi_m).
$$
\end{proof}

The translated vectors are fixed by the corresponding conjugate parahorics.

For $\mu=(m)$, one has $\mathcal L_\mu=\ell_*=0$, $\kappa=0$, and $Z(s)=L(s)$.
For $\mu=(1,\ldots,1)$ and trivial row twists,
$$
\mathcal L_\mu=\ell_*=\binom m2,\qquad
z_b=q^{b-(n+1)/2},\qquad y_b=q^{b-(m+1)/2},
$$
so the scalar specializes to
$$
\boxed{Z(s)=q^{\binom m2(s+(n+m)/2)-m(m-1)(m+1)/3}
L(s,\operatorname{St}_n\times\operatorname{St}_m).}
$$
The expression in \cite[Theorem~6]{bar-kar-iwahori} agrees with
this formula when interpreted with $\operatorname{vol}(J_m)=1$.
Keeping $\operatorname{vol}(N_m(\mathcal O))=1$, the convention stated at the start of Section~6 of our previous work
requires dividing our formula by the parahoric index
$$
[\GL_m(\mathcal O):K_\mu]
=\frac{\prod_{j=1}^m(q^j-1)}
{\prod_b\prod_{j=1}^{\mu_b}(q^j-1)}.
$$
\endgroup

\section*{Acknowledgement of AI use}
The author used AI tools to assist with editing and checking Sections~1--5,
based on the author's own ideas and prompts. ChatGPT Astra supplied the
arguments in Section~\ref{sec:rankin-selberg-gln-glm} which the author verified and simplified.

\end{document}